\documentclass[11pt]{article}
\usepackage[margin=1.15in]{geometry}

\usepackage[utf8]{inputenc}
\usepackage[T1]{fontenc}
\usepackage{authblk}
\usepackage{amsthm}
\usepackage{amssymb}
\usepackage{amsmath}
\usepackage{float}
\usepackage{tikz}
\usetikzlibrary{arrows,calc,shapes,decorations.pathreplacing}
\usepackage{hyperref}
\hypersetup{pdftex,colorlinks=true,allcolors=blue}
\usepackage{caption}
\usepackage{pgfplots}
\usepackage{subcaption}
\usepackage{enumitem}
\usepackage{bm}

\def\addlegendimage{\csname pgfplots@addlegendimage\endcsname}

\newtheorem{theorem}{Theorem}
\newtheorem{lemma}[theorem]{Lemma}
\newtheorem{proposition}[theorem]{Proposition}

\newtheorem*{theorem*}{Theorem}

\newtheorem{assumption}{Assumption}

\theoremstyle{remark}
\newtheorem{example}{Example}
\newtheorem{remark}{Remark}
\newtheorem*{remark*}{Remark}

\def\P{\mathrm{P}}
\def\E{\mathrm{E}}
\def\e{\mathrm{e}}
\def\Var{\mathrm{Var}}
\def\Cov{\mathrm{Cov}}

\def\Re{\mathrm{Re}}
\def\Im{\mathrm{Im}}

\newcommand{\apost }{'}

\newcommand{\darrow}{\:\to_{\rm d}\:}
\newcommand{\asarrow}{\:\to_{\rm as}\:}

\newcommand*\diff{\mathop{}\!\mathrm{d}}

\begin{document}

\title{Nonparametric estimation of the job-size distribution for an M/G/1 queue with Poisson sampling\footnote{To appear in Queueing Systems: Theory and Applications}}

\author{Liron Ravner}
\affil{\small{Department of Statistics, University of Haifa}}

\date{\today}
\maketitle

\begin{abstract}
This work presents a non-parametric estimator for the cumulative distribution function (CDF) of the job-size distribution for a queue with compound Poisson input. The workload process is observed according to an independent Poisson sampling process. The nonparametric estimator is constructed by first estimating the characteristic function (CF) and then applying an inversion formula. The convergence rate of the CF estimator at $s$ is shown to be of the order of $s^2/n$, where $n$ is the sample size. This convergence rate is leveraged to explore the bias-variance tradeoff of the inversion estimator. It is demonstrated that within a certain class of continuous distributions, the risk, in terms of MSE, is uniformly bounded by $C n^{-\frac{\eta}{1+\eta}}$, where $C$ is a positive constant and the parameter $\eta>0$ depends on the smoothness of the underlying class of distributions.  A heuristic method is further developed to address the case of an unknown rate of the compound Poisson input process.
\end{abstract}

\noindent\textbf{Keywords} 	Nonparametric inference ; Stochastic storage system ; Risk convergence rate ; M/G/1 queue ; Compound Poisson process ; Job-size estimation ; Inversion formula.

%%%%%%%%%%%%%%%%%%%%%%%%%%%%%%%%%%
\section{Introduction}\label{sec:intro}

Mathematical models of stochastic storage systems have a wide range of applications, such as queueing systems, insurance risk, inventory systems and communication networks (see \cite{book_P1998}). The analysis of a stochastic system's performance depends on assumptions about model primitives. These include the arrival rates, jumps in the buffer content (or workload) and processing rates. However, many systems have inherent uncertainty regarding the distributions of the underlying primitives. Statistical inference is therefore a crucial step in predicting performance analysis and subsequent design and control decisions. In many cases, the challenge lies in the indirect nature of the inference problem; the random variable of interest is not directly observed, but rather some other dependent outcome. Specifically, this work focuses on estimating the jump-size distribution without observing the actual realizations, but rather periodic observations of the system workload. An additional challenge is that observations from storage systems, such as workload observations, typically have an elaborate dependence structure, making standard statistical techniques unsuitable. Specially tailored methods are therefore often required, and the focus is often on parametric models.  Here we aim to apply state of the art nonparametric statistical analysis in order to investigate the worst-case performance of an estimation procedure that takes into account the dependence structure of the workload process.  

We consider a system where iid jobs with a continuous distribution $G$ arrive according to a Poisson process with rate $\lambda$ and are processed at a unit rate. This model is known as the canonical M/G/1 queue, although it also often used to model other processes, e.g., risk processes and communication networks. The job-size distribution is crucial in determining the system's performance. For instance, the mean sojourn time is determined by the first two moments of $G$ (see \cite{book_A2003}). However, in many practical applications, the job-size distribution is unknown and needs to be estimated from aggregate data.  This work presents a nonparametric estimator for the jump size cumulative distribution function (CDF) $G$ from discrete observations of the workload, i.e., a reflected version of the compound Poisson input process with an additional drift term.  Worst-case convergence rate guarantees are established for the risk associated with the estimator by combining asymptotic methods drawing inspiration from those applied in statistical deconvolution problems (e.g., \cite{F1991}) together with the ergodic properties of the workload process (e.g., \cite{TT1994}). 

In our setting the workload process is periodically probed by an external independent Poisson process.  The workload observations are first used to estimate the characteristic function (CF) of the job-sizes. This step relies on modifying an estimator for the characteristic exponent of the input to a L\'evy-driven queue which was given in \cite{RBM2019}. The estimator for $G$ is then obtained by applying an inversion formula. The asymptotic performance of the estimator is analysed via a mean square error risk function. For a certain class of distributions $\mathcal{G}$, the risk is shown to converge to zero at a rate of order $n^{-\eta/(1+\eta)}$, where $\eta>0$ is a smoothness parameter of $\mathcal{G}$. The $L^2$-convergence rate of the CF estimation error is established as a first step, which in turn relies on the (polynomially-ergodic) convergence rate of the workload process to its stationary distribution. As the estimator requires numerical integration, there is an inherent bias associated with any finite-sample estimator. The bias-variance tradeoff is captured via a truncation parameter for the numerical integration, which plays a crucial role in establishing the convergence rate of the risk. Finally, the estimation procedure is extended to the case where the arrival rate $\lambda$ is unknown and must be simultaneously estimated.

\vspace{0.5cm}
For a complex number $z\in\mathbb{C}$ we denote the real and imaginary parts by $\Re\{z\}$ and $\Im\{z\}$, respectively. The complex conjugate of $z$ is denoted by $\overline{z}$, and the absolute value is defined as the Euclidean distance on the complex plane, $|z|=\sqrt{\Re\{z\}^2+\Im\{z\}^2}$. For a complex valued function $g:\mathbb{R}\to\mathbb{C}$, we denote the derivative by $g'(s)=\frac{\diff}{\diff s}\Re\{g(s)\}+i\frac{\diff}{\diff s}\Im\{g(s)\}$. 

%%%%%%%%%%%%%%%%%%%%%%%%%%%%%%%%%%
\subsection{Model and objective}\label{sec:estim}

Jobs arrive at a single-server queue according to a Poisson process $(N(t))_{t\geq 0}$ with rate $\lambda$. Each job has an independent size with a common continuous distribution $\P(B\leq x)= G(x)$ and CF characteristic function $\gamma(s)=\E[\e^{is B}]$. The input process is therefore a compound Poisson process, denoted by $J(t)=\sum_{i=1}^{N(t)}B_i$. Let $X(t)=J(t)-t$ denote the net-input process. The workload process $(V(t))_{t\geq 0}$, representing the virtual waiting time of a potential arrival at time $t\geq 0$,is the net input process reflected at zero: $V(t)=X(t)+\max\{V(0),L(t)\}$, where $L(t):=-\inf_{0\leq s\leq t}X(s)$ (see \cite{book_A2003}). The characteristic exponent of the net input process is
\begin{equation}\label{eq:phi_LST}
\varphi(s):=\log\E[\e^{is X(1)}]=\lambda\left(\gamma(s)-1\right)-is\ , s\in\mathbb{R}\ .
\end{equation}
If $\rho=\lambda\E[B]<1$ the workload process has a stationary distribution that is given by the generalized Pollaczek–Khinchine (GPK) formula (e.g., \cite{DM2015})
\begin{equation}\label{eq:GPK_V}
\E[\e^{is V}]=\frac{s\P(V=0)}{\varphi(s)}\ ,  s\in\mathbb{R}\  .
\end{equation}

Suppose that the arrival rate $\lambda$ is known and $G$ is an unknown continuous distribution. The workload process is sampled according to an independent Poisson sampling process with rate $\xi>0$. Let $(T_j)_{j\geq 1}$ denote the event times of the sampling process and denote the $i$th workload observation by $V_j:=V(T_j)$. Given a sample of workload observations $(V_0,V_1,\ldots,V_n)$ the following pointwise estimator for the characteristic exponent $\varphi(s)$ was constructed in \cite{RBM2019} (and refined in \cite{NMR2024}),
\begin{equation}\label{eq:phi_z_estimator}
\widehat{\varphi}_n(s)=\frac{\frac{\xi}{n}\left(\e^{is V_n}-\e^{is V_0}\right)-\frac{is}{n}\sum_{j=1}^n\mathbf{1}\{V_j=0\}}{\frac{1}{n}\sum_{j=1}^n \e^{is V_{j}}}\ ,
\end{equation}
where $\mathbf{1}\{\}$ denotes the indicator function.  This is a Z-estimator that relies on the explicit form of the workload LST sampled at an exponentially distributed time, conditional on an initial state, which was derived in \cite{KBM2006} (see also \cite[Ch.~4.1]{DM2015}). The estimator is pointwise consistent, i.e., $\widehat{\varphi}_n(s)\asarrow \varphi(s)$ for any $s$, and the estimation errors at any collection of values $(s_1,\ldots,s_d)$ (scaled by $\sqrt{n}$) have a multivariate normal asymptotic distribution. Note that the analysis in \cite{RBM2019} was in terms of the Laplace-Stieltjes transform (LST), i.e., $\E[\e^{-s X(1)}]$ for real $s\geq 0$. The nonparametric estimator presented here relies on a inversion of the CF and will therefore require verification of some of the results for the complex domain.

We leverage the framework of \cite{RBM2019} to estimate a continuous distribution $G$ by applying the following inversion formula (see \cite[Thm.~4.4.3]{book_K1972}),
\begin{equation}\label{eq:G_inv}
G(x)=\frac{1}{2}-\frac{1}{\pi}\int_0^\infty\frac{1}{s} \Im\{\gamma(s)\e^{-i s x}\}\diff s\ .
\end{equation}
The inversion formula, utilized in deconvolution problems (e.g., \cite{DGJ2011}), serves as a natural choice in this context too, owing to the existence of an estimator for the job-size CF via the conditional workload distribution.

The estimator for the CF $\gamma$ is obtained from \eqref{eq:phi_LST} and \eqref{eq:phi_z_estimator},
\begin{equation}\label{eq:gamma_hat}
\widehat{\gamma}_n(s)=\frac{1}{\lambda}(\widehat{\varphi}_n(s)+is)+1\ , s\in\mathbb{R}\ .
\end{equation}
Then, plugging the CF estimator in the inversion formula \eqref{eq:G_inv} yields
\begin{align*}
\widehat{G}_n^{(h)}(x)=\frac{1}{2}-\frac{1}{\pi}\int_0^h\frac{1}{s} \Im\{\widehat{\gamma}_n(s)\e^{-is x}\}\diff s\ , x\geq 0\ ,
\end{align*}
where $h>0$ is a truncation parameter for the numerical integration. The estimation accuracy is determined by the value of $h$ and the tail behaviour of $\gamma$, which are directly related to the bias and variance of the estimation error. Investigating this tradeoff and its implications on the asymptotic performance of the estimator is the main objective of this work. The following example provides initial insight on this tradeoff and the estimator's performance.

%%%
\begin{example}\label{exmp:gamm_mix}\textit{Mixture of Gamma distributions:}
Suppose that the jobs sizes follow a mixture of Gamma distributions. Specifically, let $\bm{\alpha}=(\alpha_1,\ldots,\alpha_d)$, $\bm{\beta}=(\beta_1,\ldots,\beta_d)$ and $\bm{p}=(p,\ldots,p_d)$, where $\alpha_j,\beta_j,p_j>0$ for all $j=1,\ldots,d$, and $\sum_{j=1}^dp_j=1$. The job-size CDF is
\begin{align*}
G(x)=\sum_{j=1}^d p_j\int_0^x \frac{x^{\alpha_j}\e^{-\beta_j x}\beta_j^{\alpha_j}}{\Gamma(\alpha_j)}\diff x \ , x\geq 0\ ,
\end{align*}
and the corresponding CF is
\begin{equation}\label{eq:g_MG}
\gamma(s) =\sum_{j=1}^d p_j \left(\frac{\beta_j}{\beta_j-is}\right)^{\alpha_j} \ , s\in\mathbb{R} \ .
\end{equation}
This distribution is used for illustrative purposes throughout the paper as it is flexible and can approximate many continuous non-negative distributions. 
Consider now the special case with parameters $\bm{\alpha}=(2,6,1)$, $\bm{\beta}=(3.5,90,0.09)$ and $\bm{p}=(0.6,0.35,0.05)$. This is a bimodal distribution as can be seen in Figure~\ref{fig:G_density}. Figure~\ref{fig:Gh} illustrates the estimation accuracy for different choices of $h$ for a simulated sample of workload observations generated for a system with job-size distribution $G$ and arrival rate $\lambda=1$. In the example for smaller values of $h$ the fit to the true function is quite poor, while the best fit appears to be when choosing $h=5$. Interestingly, the estimator with $h=10$ yields a better fit for lower values of $x$, but appears to have greater fluctuations around the true function. This occurs because the estimation error of the CF ${\gamma}(s)$ has a higher variance as $s$ increases. In this work we will quantify this tradeoff and suggest an asymptotically optimal choice of the truncation parameter.
\hfill $\diamond$
\end{example}

 %%%
\begin{figure}[h]
\centering
\begin{subfigure}[b]{0.48\textwidth}
\includegraphics[width=\textwidth,height=6cm]{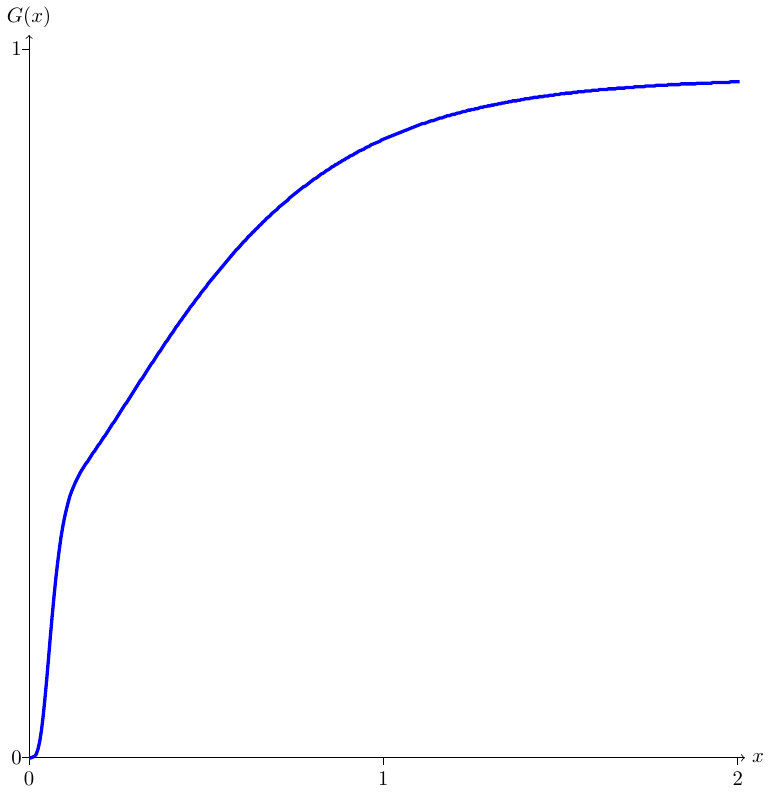}
\caption{CDF $G$.}
  \label{fig:G_cdf}
\end{subfigure}
\begin{subfigure}[b]{0.48\textwidth}
\includegraphics[width=\textwidth,height=6cm]{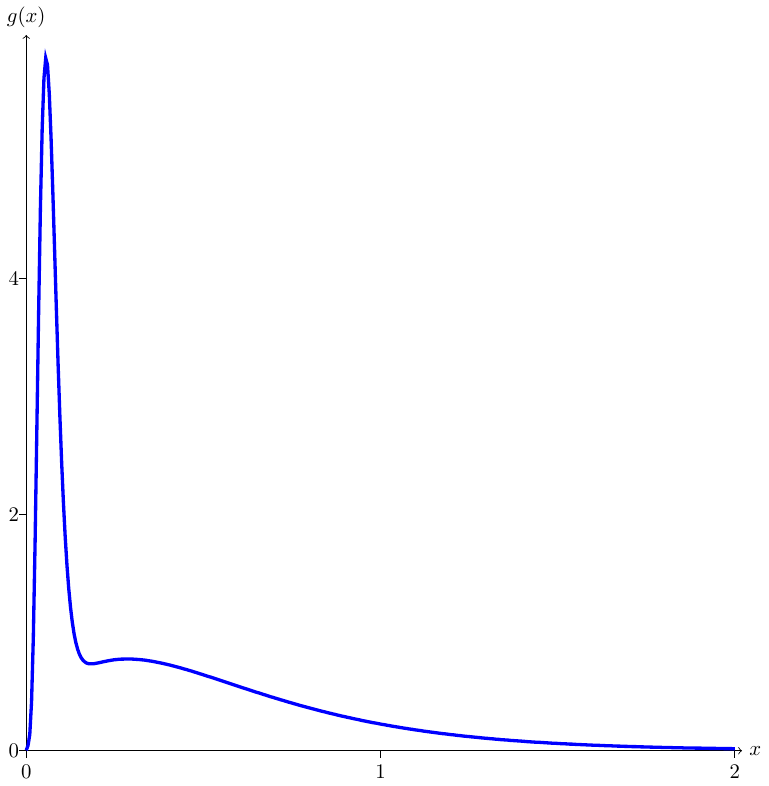}
\caption{Density $g$.}
  \label{fig:G_density}
\end{subfigure}

\caption{The CDF and density for a mixture of gamma distributions with $\bm{\alpha}=(2,6,1)$, $\bm{\beta}=(3.5,90,0.09)$ and $\bm{p}=(0.6,0.35,0.05)$.} \label{fig:G}
\end{figure}

%%%
\begin{figure}[h]
\centering
\includegraphics[width=0.8\textwidth,height=9cm]{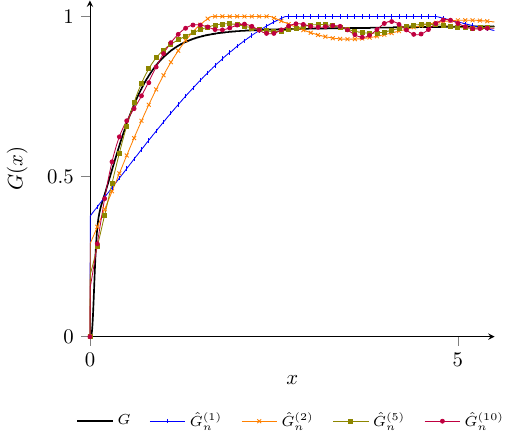}
\caption{The estimators are based on a simulation of $n=10^4$ workload observations with sampling rate $\xi=1$. The arrival rate is $\lambda=1$ and the true CDF $G$ (black solid line) is a mixture of gamma distributions $\bm{\alpha}=(2,6,1)$, $\bm{\beta}=(3.5,90,0.09)$ and $\bm{p}=(0.6,0.35,0.05)$. The estimated functions $\widehat{G}_n^{(h)}$ are plotted for $h\in\{1,2,5,10\}$.} \label{fig:Gh}
\end{figure}

Moving back to the general case, we now construct a refined version of the estimator. By \eqref{eq:phi_z_estimator} and \eqref{eq:gamma_hat}, the estimator $\widehat{\gamma}_n(s)$ can be written as
\begin{align*}
\widehat{\gamma}_n(s)=1-\frac{\frac{is}{n}\sum_{j=1}^n(\mathbf{1}\{V_j=0\}-\e^{is V_j})-\frac{\xi}{n}\left(\e^{is V_n}-\e^{is V_0}\right)}{\frac{\lambda}{n}\sum_{j=1}^n \e^{is V_{j}}}\ .
\end{align*}
Observe, however, that this estimator may not be a proper CF for some realizations of $(V_0,\ldots,V_n)$. In particular, when the denominator on the right hand side is small, then the $|\widehat{\gamma}_n(s)|$ may exceed 1 for some values of $s$. Furthermore, the fact the estimator is unbounded introduces technical difficulties in bounding the mean square error of the estimator. We overcome this by introducing a modified estimator: for some $\epsilon\in(0,1)$,
\begin{equation}\label{eq:gamma_hat_e}
\widehat{\gamma}^\epsilon_n(s)=\left\lbrace\begin{array}{cc}
\widehat{\gamma}_n(s), & \left|\frac{1}{n}\sum_{j=1}^n \e^{is V_{j}}\right|>\epsilon\ , \\
(1-\epsilon)(1+i) & \left|\frac{1}{n}\sum_{j=1}^n \e^{is V_{j}}\right|\leq \epsilon\ .
\end{array}\right.
\end{equation}
The modified estimator is less sensitive to fluctuations of the denominator term in $\widehat{\gamma}_n(s)$. Bounding the variance of these fluctuations will be crucial in establishing the uniform convergence rate of the CDF estimator for a specific class of distributions. 

The CDF estimator relying on inversion of the modified CF estimator is now defined as
\begin{equation}\label{eq:G_hat}
\widehat{G}_n^{(h)}(x)=\frac{1}{2}-\frac{1}{\pi}\int_0^h\frac{1}{s} \Im\{\widehat{\gamma}^\epsilon_n(s)\e^{-is x}\}\diff s\ , x\geq 0\ .
\end{equation} 

%%%
\begin{remark}\label{rem:epsilon1}
For an M/G/1 queue with the usual stability conditions we have that as $n\to\infty$,
\begin{align*}
 \left|\frac{1}{n}\sum_{j=1}^n \e^{is V_{j}}\right| \asarrow \left|\E[\e^{isV}]\right|\in(0,1)\ .
\end{align*}
Note that here we used the PASTA property; the limit at Poisson sampling moments coincides with the time-average limit (see \cite[Ch.ZVI]{book_A2003}). If $\left|\E[\e^{isV}]\right|>\epsilon$, then the estimator is only truncated a finite number of times, almost surely, and the asymptotic performance of the truncated estimator $\widehat{\gamma}^\epsilon_n(\cdot)$ coincides with that of the original estimator $\widehat{\gamma}_n(\cdot)$.  The choice of $(1-\epsilon)(1+i)$ ensures that the truncated CF estimator lies in $(0,1)$, although, in principle other constants may be used without affecting the asymptotic analysis.  In Section~\ref{sec:main} we further discuss how the choice of $\epsilon$ is related to specific assumptions on the class of possible job-size distributions.
\end{remark}

%%%
\begin{remark}\label{rem:decomposition}
It is important to emphasize that the inverse estimation problem deviates from the standard setting in two key aspects. Firstly, it does not involve directly inverting an empirical CF, but instead relies on the relationship between the job-size and workload distributions. Secondly, the workload observations are correlated and exhibit a distinct and specific dependence structure.
\end{remark}

Suppose that $\mathcal{G}$ is a class of cumulative distribution functions that the true distribution belongs to. Let $R(\widehat{G},G;x)=\E[\widehat{G}_n(x)-G(x)]^2 $ denote the mean square error (MSE) of $\widehat{G}$ with respect to the true distribution $G\in\mathcal{G}$. The risk of estimator $\widehat{G}_n$ at $x\geq 0$ is defined as the worst-case MSE,
\begin{equation}\label{eq:risk}
\mathcal{R}_{\mathcal{G}}(\widehat{G}_n;x) =\sup_{G\in\mathcal{G}}R(\widehat{G},G;x)= \sup_{G\in\mathcal{G}}\E[\widehat{G}_n(x)-G(x)]^2 \ .
\end{equation}
The main goal of this work is to provide conditions for which the risk converges to zero, and to provide an upper bound on the convergence rate. A key step towards asymptotic analysis of the risk is establishing a uniform in $\mathcal{G}$ convergence rate of the mean square error of the job-size CF estimator: $\E\left| \widehat{\gamma}_n^\epsilon(s)-\gamma(s)\right| ^2$ for all $s\geq 0$.

%%%%%%%%%%%%%%%%%%%%%%%%%%%%%%%%%%
\subsection{Main contributions}\label{sec:contrib}
\begin{itemize}
\item For a given arrival rate $\lambda$, and a class $\mathcal{G}$ that yields a stable queue and satisfies $\E[B^3]<\infty$, we show that the convergence rate of the mean square error of the job-size CF estimator satisfies
\begin{align*}
\E\left| \widehat{\gamma}_n^\epsilon(s)-\gamma(s)\right| ^2 \approx\frac{s^2}{n}\ \ , \forall s\geq 0\ , \forall G\in\mathcal{G} \ ,
\end{align*}
where $a_n\approx b_n$ means $\lim_{n\to\infty}a_n/b_n=C$ for some constant $C$. This shows that the CF estimator has a parametric rate, but the rate is scaled by a quadratic term with respect to the CF variable $s$. This is a critical step in establishing the nonparemetric convergence rate of the risk.  The reason for this is that computing the CDF estimator $\widehat{G}_n^{(h)}$ defined in \eqref{eq:G_hat} entails integration of $\widehat{\gamma}_n^\epsilon(s)$ until some truncation parameter $h$. Moreover, this result confirms a numerical observation of \cite{RBM2019} that the asymptotic variance of the characteristic exponent $\varphi(s)$ estimation error increases with $s$ at a quadratic rate.
\item For a given arrival rate $\lambda$, by setting an appropriate sequence of truncation parameters $(h_n)_{n\geq 1}$, a uniform upper bound for the convergence rate of the risk is established. In particular, assuming that $\mathcal{G}$ is class of continuous functions with a CF that vanishes in absolute value to zero at least at a polynomial rate $s^{-\eta}$, 
\begin{align*}
\mathcal{R}_{\mathcal{G}}(\widehat{G}_n^{(h_n)};x)= O\left(n^{-\frac{\eta}{1+\eta}}\right) \ , \forall x\geq 0\ ,
\end{align*}
where $h_n=n^{1/(2(1+\eta))}$. This further provides a practical way to select the truncation parameter $h$ for a given sample size. The class of functions $\mathcal{G}$ includes the ordinary-smooth and super-smooth distributions in the terminology of the statistical deconvolution literature (see \cite{F1991}).  Examples of such distributions are Gamma, log-normal, truncated normal, pareto and their mixtures.  Note that this result provides an upper bound on the risk, but does not guarantee rate-optimality, in the sense of \cite{book_T2009}. Establishing a tight lower bound on the risk seems to be a very challenging problem due to the intractability of the joint distribution of the workload observations. Nevertheless, this work is the first to present any type of guarantee on the rate for nonparametric estimation of $G$. 
\item For the proofs of the main results, several auxiliary results pertaining to the autocorrelation function of the workload process sampled at Poisson epochs are established.
\item For the scenario in which the arrival rate is unknown, a heuristic estimation scheme that simultaneously estimates the arrival rate $\lambda$ and job-size distribution $G$ is constructed. Simulation experiments suggest that for large sample sizes the performance of the scheme is very close to that of the benchmark that uses the known $\lambda$.
\item The theoretical results are supplemented with simulation analysis and a discussion of practical implementation challenges. In particular, we examine the sensitivity of the procedure to the truncation parameter $h$. While in the asymptotic regime, $h_n$ must tend to infinity for the bias to vanish, we observe that for finite samples, moderate values of $h$ outperform larger ones. This occurs because the variance of the CF estimator is $O(s^2/n)$, so increasing $h$ incorporates noisier terms into the integral of the inversion formula.  An interesting observation in these experiments is that the estimation accuracy is better when $\rho$ is lower, but the optimal choice of truncation parameter remains more or less the same. The same is observed for a slower sampling rate, which also reduces the variance of estimation errors because the correlation between workload observations is weaker.
\end{itemize}

%%%%%%%%%%%%%%%%%%%%%%%%%%%%%%%%%%
\subsection{Related literature}\label{sec:lit}

Classical queueing theory deals with probabilistic analysis of congested systems and with their corresponding performance such as waiting times. A complementary stream in the literature is focused on the development of statistical inference methods for queueing systems where some of the underlying parameters are unknown. This section reviews a small subset of this literature with the goal of highlighting the context and relevance of the current work. For a recent in depth survey of the literature on statistical inference the reader is referred to \cite{ANT2021}.

The first works on this topic naturally focused on parametric methods, some of which are surveyed in \cite{BR1987}. For example, \cite{BP1988} considered a single-server queue that is observed for some fixed duration of time. This results in a random sample of arrival and service times following some parametric distributions, for which an MLE is computed. The large sample asymptotic performance of the estimator is further studied. In \cite{BBL1996} an MLE is derived for a G/G/1 queue with waiting time observations by assuming the increment in Lindley's recursive formula follows a parametric distribution. For estimators based on waiting time (or queue length) data, asymptotic analysis of the estimation error must take into account the covariance structure of the process (see \cite{R1975}), and this will also play a key role in the analysis presented here. Another key feature in statistical inference for queueing systems is that it typically involves indirect estimation. For example, in \cite{BP1988} waiting times observations are treated as a censored version of the increments in Lindley's formula, i.e., censoring occurs when the waiting time is zero. 

The challenges of indirect estimation related to probing data are considered in detail in \cite{BKV2009}, and estimators for specific models are discussed. A nonparametric estimator for $G$ based on discrete workload observations is provided in \cite{HP2006}. Specifically, Pollaczek–Khinchine's formula is inverted to obtain an estimator of the residual job-size distribution from the empirical workload distribution. However, estimating $G$ from its residual counterpart requires a heuristic step that is shown to be very noisy for systems with very high or very low load. Nonparametric inference of the job-size distribution is mostly carried out for infinite server queues and networks (e.g., \cite{BP1999}, \cite{SW2015} \cite{G2016} and \cite{GK2019}). %For a recent survey of the statistical problems and techniques in a wide range of queueing systems, see \cite{ANT2021}.

The problem of decompounding a compound Poisson sum was studied in \cite{BG2003}. In this setting the goal is to infer $G$ by observing realizations of $\sum_{i=1}^{N(t)}B_i$. Inversion techniques that have been applied for this problem are explored in \cite{vEGS2007} and a Bayesian approach is presented in \cite{GMvdM2020}. Our setting is distinct from the decompounding problem in that we do not observe the compound Poisson process, but rather a process that includes a drift term and is reflected at zero. An inversion estimator for a similar process without reflection, namely a finite intensity subordinator (compound Poisson process with drift), that is observed at discrete time points is given in \cite{G2012}.  A rate optimal density estimator is derived in \cite{D2013} for the jumps of a discretely observed compound Poisson process (without reflection). Estimation of the characteristic exponent of general L\'evy processes from discrete (low frequency) observations is further studied in \cite{CG2009}, \cite{NR2009} and \cite{BR2015}.
The methods applied here are also related to those of nonparametric noise deconvolution problems. In particular, the rate of convergence of the estimator is related to the smoothness class of the job-size distribution (see \cite{F1991} and \cite{DGJ2011}).

%%%%%%%%%%%%%%%%%%%%%%%%%%%%%%%%%%
\subsection{Organization of the paper}\label{sec:organization}

Section~\ref{sec:main} presents the main results of the paper, along with an outline of their proofs and a discussion of their implications. Section~\ref{sec:proof} provides some auxiliary results on the autocorrelation function of the workload and the necessary proofs for the propositions leading up to the main results. In Section~\ref{sec:lambda}, a modified estimator is constructed for the case of an unknown arrival rate. Finally, Section~\ref{sec:conclusion} provides concluding remarks on possible extensions and modifications of this work.

%%%%%%%%%%%%%%%%%%%%%%%%%%%%%%%%%%
\section{Main results}\label{sec:main}

Recall that the risk of an estimator $\widehat{G}$ was defined in \eqref{eq:risk} as the worst-case MSE for a given class of job-size distributions $\mathcal{G}$,
\begin{align*}
\mathcal{R}_{\mathcal{G}}(\widehat{G};x) =\sup_{G\in\mathcal{G}}R(\widehat{G},G;x)\ 
\end{align*}
where
\begin{align*}
 R(\widehat{G},G;x)=\frac{1}{\pi^2}\E\left[\int_0^h\frac{1}{s} \Im\{(\widehat{\gamma}_n^\epsilon(s)-\gamma(s))\e^{-is x}\}\diff s-\int_h^\infty\frac{1}{s} \Im\{\gamma(s)\e^{-is x}\}\diff s\right]^2\ .
\end{align*}
The MSE can be decomposed into variance and bias terms, with an additional error term that will be negligible in the asymptotic analysis. Specifically,
\begin{align*}
 R_G(\widehat{G}_n^{(h)};x) =\frac{1}{\pi^2}(v_x+b_x^2-e_x)\ ,
\end{align*}
where 
\begin{align*}
v_x &:= \E\left[\int_{0}^h\int_{0}^h\frac{1}{st} \Im\{(\widehat{\gamma}_n^\epsilon(s)-\gamma(s))\e^{-is x}\}\Im\{(\widehat{\gamma}_n^\epsilon(t)-\gamma(t))\e^{-it x}\}\diff s \diff t\right] \ , \\
b_x&:=\int_h^\infty\frac{1}{s} \Im\{\gamma(s)\e^{-is x}\}\diff s\ , \\
e_x &:= 2 b_x\E\left[\int_0^h\frac{1}{s} \Im\{(\widehat{\gamma}_n^\epsilon(s)-\gamma(s))\e^{-is x}\}\diff s \right] \ .
\end{align*}
Standard arguments further yield
\begin{equation}\label{eq:risk_up}
 R_G(\widehat{G}_n^{(h)};x) \leq \frac{2}{\pi^2}(v_x+b_x^2)\ .
\end{equation}

The main result of the paper is providing a uniform upper bound on the rate of convergence of the risk in a class of distributions $\mathcal{G}$ that we define below.
The first step will be to establish the $L^2$-convergence rate of $\widehat{\gamma}_n^\epsilon(s)$ to the true CF $\gamma(s)$, as a function of $s$. This is then used in order to construct an increasing sequence of truncation parameters $(h_n)_{n\geq 1}$ that ensures that both the bias ($b_x$) and variance ($v_x$) terms in the risk converge to zero at the same rate. The rate itself will depend on a smoothness parameter associated with the class $\mathcal{G}$. The high-level outline of the proofs will be provided in Section~\ref{sec:out_proof}. The proofs rely on two propositions pertaining to the convergence rate of the terms that appear in the estimation equation. The proofs of the propositions, along with some necessary lemmas, will be given in Section~\ref{sec:proof}.

%%%
\begin{assumption}\label{assum:G_class}
Let $\mathcal{G}$ denote a class of continuous job-size distributions satisfying the following conditions.
\begin{enumerate}
\item[(a)] There exists a constant $\delta\in(0,\frac{1}{2})$ such that
 \begin{equation}\label{eq:assum_rho}
\lambda\E[B]< 1-\delta \ , \forall G\in\mathcal{G}\ .
 \end{equation}
\item[(b)] There exists a constant $0<M<\infty$ such that
 \begin{equation}\label{eq:assum_moments}
\E[B^3]<M\ , \forall G\in\mathcal{G}\ .
 \end{equation}
\item[(c)] There exist constants $\eta>0$ and $0<C_0<\infty $ such that 
 \begin{equation}\label{eq:assum_smooth}
 \lim_{s\to\infty} |\gamma(s)|s^{\eta}\leq C_0\ , \forall G\in\mathcal{G}\ .
 \end{equation}
\end{enumerate}
\end{assumption}

%%%
\begin{assumption}\label{assum:stationary}
The workload sequence $(V_j)_{j\geq 0}$ is stationary in the strong sense; for any $k\geq 1$,
\begin{align*}
(V_0,\ldots,V_{j})\sim (V_{k},\ldots,V_{k+j})\ , \forall j\geq 1 \ .
\end{align*}
\end{assumption}

Assumption~\ref{assum:stationary} requires that the sampled workload be stationary. This implies two properties that will be used in the proofs of the main theorems. Specifically, all samples share the same marginal distribution:
\begin{align*}
    \E[\e^{is V_j}]& =\E[\e^{is V_k}]\ , \forall j,k\geq 0\ , s\in\mathbb{R} \ .
\end{align*}
Additionally, strong-sense stationary implies wide-sense stationary, i.e., the autocorrelation function is invariant to the sampling index:
\begin{align*}
    \Cov(V_{j},V_{k}) &= \Cov(V_{0},V_{k-j})\ , \forall k\geq j\geq 1\ .
\end{align*}
Since the process is regenerative, it has a stationary version (see \cite{T1992}), which ensures the existence of a probability space consistent with Assumption~\ref{assum:stationary}. From a practical perspective, this can be interpreted as assuming that the sampled system has been operating for a long time—recalling that a regenerative process converges to its stationary distribution.

Assumption~\ref{assum:G_class} imposes three conditions on the job-size distribution: (a) First of all, a stability condition for the supremum in $\mathcal{G}$ of the utilization rates $\rho= \lambda\E[B]$ to be smaller than $1-\delta$, ensuring that the sampled system cannot be too close to critical. (b) That the third moments of job-size distributions in $\mathcal{G}$ are uniformly bounded by a constant $M$, which will be necessary for the mixing rate of the covariance terms associated with the MSE. (c) A smoothness condition that will be used to bound the bias term. Specifically, \eqref{eq:assum_smooth} ensures that $|b_x|\leq C_0 h^{-\eta}/(\eta)$ as $h$ grows. Therefore, the square of the deterministic bias term goes to zero with rate $h^{-2\eta}$ as $h\to\infty$, and this will be balanced with the convergence rate of the variance term. Note that $\lim_{s\to\infty}|\gamma(s)|=0$ for any continuous distribution (see \cite[Prop.~10.1.2]{book_AL2006}), and the smoothness of the distribution, represented by the parameter $\eta$ in our setting, determines the speed of convergence. In the context of estimation error deconvolution, distributions that are ordinary-smooth and super-smooth satisfy this condition (see \cite{F1991,DGJ2011}). Some examples of distributions satisfying Assumption~\ref{assum:G_class}(c) are exponential, Gamma, Normal, log-Normal and mixtures of such distributions. Suppose, for instance, that $G$ is as a mixture of Gamma distributions as detailed in Example~\ref{exmp:gamm_mix}. Then, by \eqref{eq:g_MG} we have that \eqref{eq:assum_smooth} holds for any $\eta \leq \min\{\alpha_1,\ldots,\alpha_d\}$. Note that there are distributions with a faster than polynomial decay rate of $|\gamma(s)|$, such as the absolute value of a normal distribution which has exponential decay. Thus, \eqref{eq:assum_smooth} holds for such distributions for any $\eta>0$. As far as the author is aware, there are no standard continuous distributions for which there exists no $\eta>0$ such that \eqref{eq:assum_smooth} is satisfied, i.e., ones with slower than polynomial decay rate. An example for a continuous distribution that violates this condition is one with an uncountable number of atoms on a finite interval, e.g., a distribution on the Cantor set which is not absolutely continuous. Hence, Assumption~\ref{assum:G_class}(c) rules out such distributions.

%%%
\begin{theorem}\label{thm:gamma_s_rate}
Suppose Assumptions~\ref{assum:G_class}(a),~\ref{assum:G_class}(b) and \ref{assum:stationary} are satisfied. If $\epsilon<\delta/4$ then there exist constants $0<C_1,C_2<\infty$ such that for any $G\in\mathcal{G}$,
\begin{equation}\label{eq:gamma_s_rate}
\frac{C_1\max\{1,s^2\}}{n}\leq \E\left| \widehat{\gamma}^\epsilon(s)-\gamma(s)\right| ^2\leq \frac{C_2\max\{1,s^2\}}{n}\ , \forall s\geq 0, \ n\geq 1 \ .
\end{equation}
\end{theorem}

%%%
\begin{theorem}\label{thm:rate}
Suppose Assumptions~\ref{assum:G_class}-\ref{assum:stationary} are satisfied. Let $\epsilon<\delta/4$ and $h_n=n^{\zeta}$, for all $n\geq 1$, where $\zeta=(2(1+\eta))^{-1}$ and $\eta$ is the smoothness parameter in \eqref{eq:assum_smooth}. Then, there exists a constant $0<C_3<\infty$ such that for any $n\geq 1$,
\begin{equation}\label{eq:risk_rate}
\mathcal{R}_{\mathcal{G}}(\widehat{G}_n^{(h_n)};x)=\sup_{G\in\mathcal{G}}\E[\widehat{G}_n^{(h_n)}(x)-G(x)]^2\leq C_3 n^{-\frac{\eta}{1+\eta}}\ , \forall x\geq 0\ .
\end{equation}
\end{theorem}

%%%
\begin{remark}\label{rem:eps} The stationary workload distribution has an atom at zero, hence $\lim_{s\to\infty}|\E[\e^{is V}]|=\P(V=0)>0$. Therefore, assuming Assumption~\ref{assum:G_class}(a), we can select a parameter $\epsilon\leq \delta/4<\P(V=0)/4\leq\inf_{s\geq 0}|\E[\e^{is V}]|/2$ for the modified estimator $\widehat{\gamma}_\epsilon$ defined in \eqref{eq:gamma_hat_e} (a fact that will later be verified in Lemma~\ref{lemma:phi_bound}). From a practical perspective $\epsilon$ can be chosen to be arbitrarily small, and there is a simple empirical test to verify this choice, namely $\epsilon<\frac{1}{4n}\sum_{j=1}^n\mathbf{1}\{V_j=0\}$.
\end{remark} 

%%%
\begin{remark}\label{rem:stationary} The bounds in Theorems~\ref{thm:gamma_s_rate} and~\ref{thm:rate} are not a function of the sampling rate $\xi$. Considering $\xi\to 0$ implies that the same convergence rates hold for an iid sample of stationary workload observations. Note that in this case the charechteristic exponent estimator given in \eqref{eq:phi_z_estimator} is
\begin{align*}
\widehat{\varphi}_n(s)=\frac{-\frac{is}{n}\sum_{j=1}^n\mathbf{1}\{V_j=0\}}{\frac{1}{n}\sum_{j=1}^n \e^{is V_{j}}}\ ,
\end{align*}
which is simply the solution of equating the empirical workload CF with the empirical GPK formula \eqref{eq:GPK_V}.  
\end{remark} 

%%%%%%%%%%%%%%%%%%%%%%%%%%%%%%%%%%%
\subsection{On the risk convergence rate}\label{sec:examples}

Theorem~\ref{thm:rate} provides an upper bound on the risk for distributions satisfying Assumption~\ref{assum:G_class}. The convergence rate is determined by the smoothness parameter $\eta$ in \eqref{eq:assum_smooth};
\begin{align*}
\mathcal{R}_{\mathcal{G}}(\widehat{G}_n^{(h_n)};x)\leq C_3 n^{-\frac{\eta}{1+\eta}}= C_3 n^{-\frac{1}{1+1/\eta}} \ .
\end{align*}
If $\eta$ is very high, i.e. the distribution is `very smooth', then the convergence rate is almost parametric ($1/n$). In general, the value $1/\eta$ can be thought of as a penalty to the parametric rate that is higher as the distribution is `less smooth'. It is interesting to point out that very similar convergence rates, i.e., polynomial that are close to parametric when $\eta$ is low, are shown to be rate-optimal in the setting of kernel density estimation with additive observation errors that satisfy condition \eqref{eq:assum_smooth} (e.g., \cite[Thm.~2]{F1991}). Of course, in practice $\eta$ is not known and setting a specific value for $h_n=n^{1/(2(1+\eta))}$ represents some prior belief on the smoothness of the possible distributions. In particular, choosing a low value of $\eta$ can be considered conservative. In the next examples we illustrate this interplay for specific distributions. The examples highlight the fact that as the underlying smoothness of $G$, indexed by the parameter $\eta$, increases, the estimation accuracy improves and a lower truncation parameter $h$ is required.  The following simulation experiments suggest that selecting a truncation parameter that yields a smoother estimated CDF is a reasonable heuristic for choosing $h$ when $\eta$ is unknown.

\setcounter{example}{0}
%%%
\begin{example}\label{exmp:gamm_mix}\textit{ (Continued) Mixture of Gamma distributions:}
In Section~\ref{sec:estim} the estimator was applied with different choices of $h$ to a mixture of Gamma distributions with $\bm{\alpha}=(2,6,1)$, $\bm{\beta}=(3.5,90,0.09)$ and $\bm{p}=(0.6,0.35,0.05)$. Theorem~\ref{thm:rate} prescribes selecting the minimal shape parameter $\eta=\min\bm\alpha=1$, yielding $\zeta=1/4$ and $h_n\approx 10$. In Figure~\ref{fig:Gh} the estimated functions are plotted for a sample size of $n=10^4$. It was observed in the example that the best fit appears to be in the range of $[5,10]$. Keep in mind that the rates in Theorem~\ref{thm:rate} are valid up to scaling by constants so $h_n\approx 10$ is certainly in line with the empirical observations for this example. The MSE upper bound \eqref{eq:risk_rate} for $n=10^4$ is of of the order $n^{-1/2}=0.01$ for $h_n=10$. Of course, these bounds are scaled by constants which may still be significant for this sample size.

In practice, the underlying distribution and $\eta$ itself are unknown, so the above is not a piratical recipe for choosing an appropriate $h$. A useful observation is that that the estimated function with $h=10$ displays much more moderate fluctuations compared to all other choices, thus suggesting that this is a reasonable heuristic criteria for the choice of $h$. The same behavior is observed in the following two examples, further strengthening the usefulness of this heuristic.
\hfill $\diamond$
\end{example}

%%%
\begin{example}\textit{Log-normal distribution:}\label{exmp:log-normal} 
Suppose now that the job-sizes follow a log-Normal distribution with parameters $(\mu,\sigma)$; $G(x)=\Phi((\log x-\mu)/\sigma)$. The characteristic function does not admit a closed form expression but via approximations it can be shown that it decays in absolute value at a polynomial rate with $\eta>2$ (see \cite{AJR2016}). For the case of $\lambda=0.6$ $\mu=0.2$, and $\sigma=0.5$ the estimator was applied for different $h$ for a sample of $n=10^4$ workload observations. Hence, by Theorem~\ref{thm:rate} we expect that $h_n=n^{1/6}\approx 4.6$ will yield a good fit. The results are illustrated in Figure~\ref{fig:Gh_ln}. Indeed, the best fit is achieved with $h=4$, while increasing the truncation threshold yields a noisier estimator. 

Figure~\ref{fig:Gh_ln2a} displays the same experiment for data generated with a lower arrival rate of $\lambda=0.3$ (half the utilization rate). This reduces the variance of the estimators and yields CDF estimators with less fluctuations in all cases. As before, the truncation parameter of $h=4$ appears to yield the best fit, although now the difference between the estimated functions is smaller due to the reduced variance. A similar experiment is presented in Figure~\ref{fig:Gh_ln2b}, but this time the sampling rate is slower ($\xi=0.2$ instead of $\xi=1$). Again, the optimal choice of truncation is obtained at $h=4$ with slightly lower variance in all cases. The reason for the reduced variance is that the correlation between workload observations is weaker when sampling occurs less frequently. See \cite{RBM2019} for a more detailed analysis of the effect of the sampling rate on the variance of estimation errors. Note that a slower sampling rate comes at a cost of a longer sampling duration for the same number of observations.
\hfill $\diamond$
\end{example}

%%%
\begin{figure}[h]
\centering
\includegraphics[width=0.8\textwidth,height=9cm]{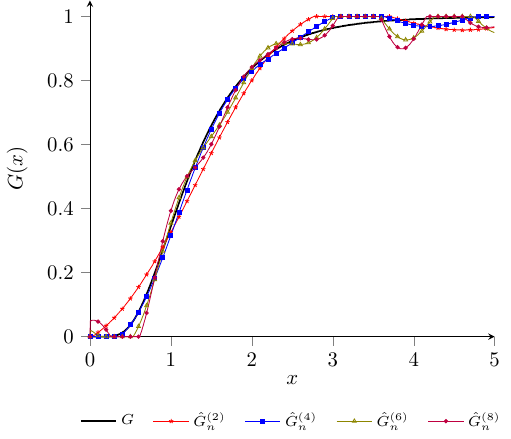}
\caption{The estimators are based on a simulation of $n=10^4$ workload observations with sampling rate $\xi=1$. The arrival rate is $\lambda=0.6$ and the true CDF $G$ (black solid line) is a log-Normal distribution with $\mu=0.2$ and $\sigma=0.5$. The estimated functions $\widehat{G}_n^{(h)}$ are plotted for $h\in\{2,4,6,8\}$.} \label{fig:Gh_ln}
\end{figure}

 %%%
\begin{figure}[h]
\centering
\begin{subfigure}[b]{0.48\textwidth}
\includegraphics[width=\textwidth,height=6cm]{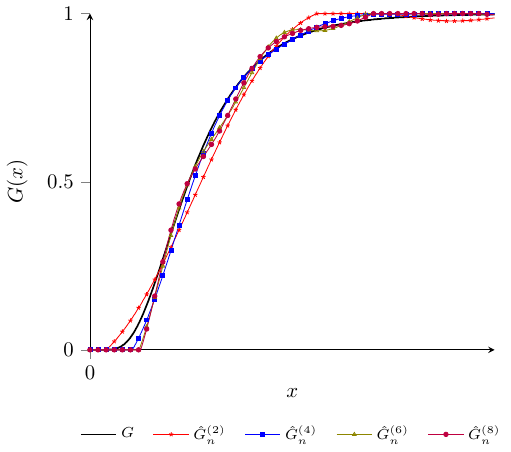}
\caption{$\lambda=0.3$, $\xi=1$.}
  \label{fig:Gh_ln2a}
\end{subfigure}
\begin{subfigure}[b]{0.48\textwidth}
\includegraphics[width=\textwidth,height=6cm]
{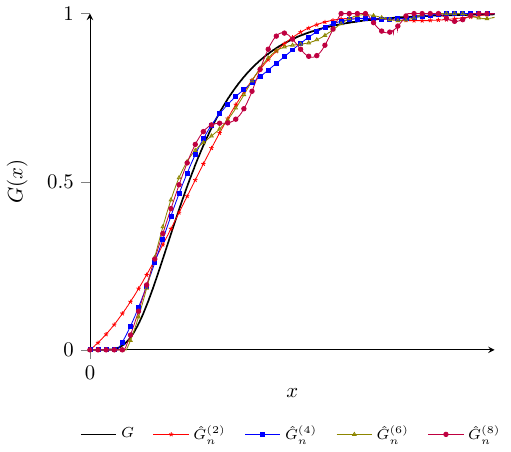}
\caption{$\lambda=0.6$, $\xi=0.2$.}
  \label{fig:Gh_ln2b}
\end{subfigure}
\caption{The estimators are based on a simulation of $n=10^4$ workload observations. The plots illustrate the impact of changing the utilization rate and sampling rate. The true CDF $G$ (black solid line) is a log-Normal distribution with $\mu=0.2$ and $\sigma=0.5$. The estimated functions $\widehat{G}_n^{(h)}$ are plotted for $h\in\{2,4,6,8\}$..} \label{fig:Gh_ln2}
\end{figure}

%%%
\begin{example}\textit{Truncated-normal distribution:}\label{exmp:normal} 
We next consider a system with jobs sizes distributed as the absolute value of normal random variables with parameters $(\mu,\sigma)$; $B=|\mathcal{N}(\mu,\sigma^2)|$. The characteristic function of $B$ decays exponentially fast in absolute value. Therefore, condition \eqref{eq:assum_smooth} is satisfied for any $\eta>0$. This means that $h_n=n^\zeta$ can be chosen with a very small $\zeta>0$ and the convergence rate will be close to $(1/n)$. In Figure~\ref{fig:Gh_norm} the estimators are plotted for different values of $h$ for a sample of $n=10^4$ workload observations.  A choice of $h\in[1,2]$ appears to yield the best fit in this case. 
\hfill $\diamond$
\end{example}

%%%
\begin{figure}[h]
\centering
\includegraphics[width=0.8\textwidth]{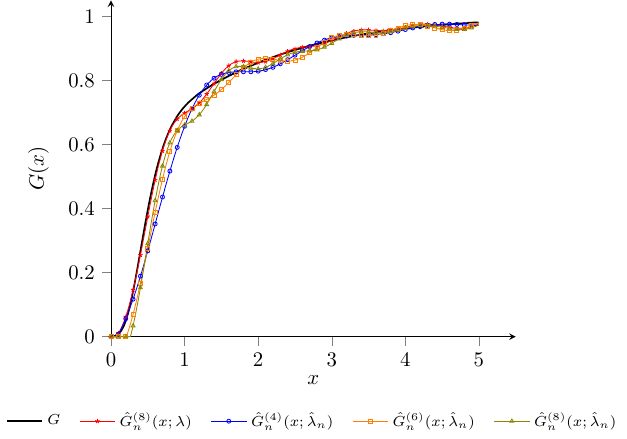}

\caption{The estimators are based on a simulation of $n=10^4$ workload observations with sampling rate $\xi=1$. The arrival rate is $\lambda=0.6$ and the true CDF $G$ (black solid line) is a truncated-normal distribution with $\mu=0.5$ and $\sigma=0.1$. The estimated functions $\widehat{G}_n^{(h)}$ are plotted for $h\in\{1,2,4,6\}$.} \label{fig:Gh_norm}
\end{figure}

\pagebreak

%%%%%%%%%%%%%%%%%%%%%%%%%%%%%%%%%%%
\subsection{Outline of proofs}\label{sec:out_proof}

For any $s\geq 0 $, let
\begin{equation}\label{eq:Zi_s}
 Z_j(s)=(\xi-\varphi(s))\e^{is V_{j}}-\xi\e^{is V_{j-1}}-is\mathbf{1}\{V_j=0\}\ .
\end{equation} 
Rearranging the terms in the estimation equation \eqref{eq:phi_z_estimator} yields
\begin{equation}\label{eq:error_Z}
\lambda(\widehat{\gamma}_n(s)-\gamma(s))=\widehat{\varphi}_n(s)-\varphi(s)=\frac{\sum_{j=1}^nZ_j(s)}{\sum_{j=1}^n\e^{is V_j}} \ ,
\end{equation}
whenever $|\frac{1}{n}\sum_{j=1}^n\e^{is V_j}|\geq \epsilon$. This representation of the estimation error is central in the proof of Theorem~\ref{thm:gamma_s_rate}. The advantage lies in the fact the numerator is a martingale difference series and the denominator is the empirical characteristic function of the workload process.  Before presenting the proofs of the theorems, the following propositions establish the $L^2$-convergence rates of both the series in \eqref{eq:error_Z}. The proofs of the propositions, along with a some necessary lemmas that lead up to them, are given in Section~\ref{sec:proof}. The proof of Theorem~\ref{thm:rate} then follows by applying Theorem~\ref{thm:gamma_s_rate} and standard inequalities.

%%%
\begin{proposition}\label{prop:cf_cov_rate}
Assuming \ref{assum:G_class} and \ref{assum:stationary}, for any $s\geq 0$ and $\epsilon<\delta/4$ there exists a positive and finite constant $D_1$ such that
\begin{equation}\label{eq:CF_norm_bound}
\sup_{G\in\mathcal{G}}\P\left(\left|\frac{1}{n}\sum_{j=1}^n\e^{is V_j}\right|\leq \epsilon\right)\leq \frac{D_1}{n}\ .
\end{equation}
\end{proposition}

%%%
\begin{proposition}\label{prop:Z_cov_rate}
Assuming \ref{assum:G_class} and \ref{assum:stationary}, there exist positive and finite constants $D_2,D_3$ such that for any $G\in\mathcal{G}$, 
\begin{equation}\label{eq:Z_nprm_bound}
\frac{D_2 \max\{1,s^2\}}{n}\leq \E|\frac{1}{n}\sum_{j=1}^nZ_j(s)|^2 \leq \frac{D_3 \max\{1,s^2\}}{n} \ , \forall s\geq 0, \ n\geq 1 \ .
\end{equation}
\end{proposition}

%%%
\begin{proof}[Proof of Theorem~\ref{thm:gamma_s_rate}]
By \eqref{eq:phi_LST} and \eqref{eq:error_Z}, the estimation error of the estimator without truncation is
 \begin{align*}
\lambda(\widehat{\gamma}_n(s)-\gamma(s))=\widehat{\varphi}_n(s)-\varphi(s)=\frac{\frac{1}{n}\sum_{j=1}^nZ_j(s)}{\frac{1}{n}\sum_{j=1}^n\e^{is V_j}} \ .
 \end{align*}
Let $A_\epsilon=\left\lbrace\left|\frac{1}{n}\sum_{j=1}^n\e^{is V_j}\right|> \epsilon\right\rbrace$, then
\begin{align*}
\E\left| \widehat{\gamma}_n^\epsilon(s)-\gamma(s)\right|^2 &= \E\left|(\widehat{\gamma}_n(s)-\gamma(s))\mathbf{1}\{A_\epsilon\}\right|^2+\E\left| (\widehat{\gamma}_n^\epsilon(s)-\gamma(s))\mathbf{1}\{\overline{A_\epsilon}\}\right|^2 \\
&\leq \E\left| \frac{\frac{1}{n}\sum_{j=1}^nZ_j(s)}{\lambda\epsilon}\right|^2+4\left(1-\epsilon\right)^2\P\left(\left|\frac{1}{n}\sum_{j=1}^n\e^{is V_j}\right|\leq \epsilon\right) \  .
\end{align*}
The upper bound in \eqref{eq:gamma_s_rate} then follows by applying \eqref{eq:CF_norm_bound} and \eqref{eq:Z_nprm_bound} in Propositions~\ref{prop:cf_cov_rate} and~\ref{prop:Z_cov_rate}.

For the lower bound in \eqref{eq:gamma_s_rate} observe that $|\frac{1}{n}\sum_{j=1}^n\e^{is V_j}|\leq 1$, hence
\begin{align*}
\E| \widehat{\gamma}_n^\epsilon(s)-\gamma(s)|^2\geq \frac{1}{\lambda^2}\E| \frac{1}{n}\sum_{j=1}^nZ_j(s)|^2 \ , 
\end{align*}
and the result follows from the lower bound in \eqref{eq:Z_nprm_bound}.
\end{proof}

%%%
\begin{proof}[Proof of Theorem~\ref{thm:rate}]
Let $(h_n)_{n\geq 1}$ be a sequence of truncation parameters such that $h_n\to\infty$ as $n\to\infty$. We first bound the convergence rates for the two terms of the risk in \eqref{eq:risk_up}.
\begin{enumerate}
\item \textit{Bound for $v_x$:} If $|\frac{1}{n}\sum_{j=1}^n\e^{is V_j}|< \epsilon$ then $|\widehat{\gamma}^{(\epsilon)}(s)-\gamma(s)|\leq 1+\sqrt{2}$ for any $s\geq 0$ by the estimator construction in \eqref{eq:gamma_hat_e}. By \eqref{eq:Zi_s} and \eqref{eq:error_Z}, if $|\frac{1}{n}\sum_{j=1}^n\e^{is V_j}|\geq \epsilon$ then 
\begin{align*}
\lambda|\widehat{\gamma}^{(\epsilon)}(s)-\gamma(s)|=\frac{|\frac{1}{n}\sum_{j=1}^nZ_j(s)|}{|\frac{1}{n}\sum_{j=1}^n\e^{is V_j}|} \leq \frac{\frac{1}{n}\sum_{j=1}^n|Z_j(s)|}{\epsilon} \ .
\end{align*}
By the triangle inequality, $|Z_j(s)|\leq |\varphi(s)|+s+2\xi$, hence there exist constants $s_0,c_0>0$ such that $|Z_j(s)|\leq c_0(|\varphi(s)|+s)$ for all $s\geq s_0$. Furthermore, as $|Z_j(s)|$ is continuously differentiable and $Z_j(0)=0$, we have that 
\begin{align*}
    |Z_j(s)|\leq K_0(|\varphi(s)|+s)\ , \forall s\geq 0
\end{align*}
where $K_0=\max\left\lbrace c_0,\sup_{s\in[0,s_0]}\frac{\diff}{\diff s}|Z_j(s)|\right\rbrace<\infty$. Therefore,
\begin{align*}
\lambda|\widehat{\gamma}^{(\epsilon)}(s)-\gamma(s)| \leq \frac{K_0(|\varphi(s)|+s)}{\epsilon} \ ,
\end{align*}
Let $K_1=K_0/(\lambda\epsilon)$. Applying the triangle inequality once more yields
\begin{align*}
  |\Im\{(\widehat{\gamma}_n^\epsilon(s)-\gamma(s))\e^{-is x}\}| \leq |(\widehat{\gamma}_n^\epsilon(s)-\gamma(s))\e^{-is x}|
=    |\widehat{\gamma}_n^\epsilon(s)-\gamma(s)| \leq K_1(|\varphi(s)|+s)\ .
\end{align*}
Therefore, for any $h>0$,
\begin{equation}\label{eq:fubini_bound}
\int_{0}^h|\frac{1}{s} \Im\{(\widehat{\gamma}_n^\epsilon(s)-\gamma(s))\e^{-is x}\}|\diff s \leq \int_{0}^h\frac{1}{s}K_1(|\varphi(s)|+s)\diff s <\infty \ .
\end{equation}
The finiteness of the last integral is verified by observing that
\begin{align*}
 \lim_{s\downarrow 0} \frac{|\varphi(s)|}{s}=|\lim_{s\downarrow 0} \frac{\varphi(s)}{s}|=|\varphi'(0)|<\infty \ .
\end{align*}

Utilizing \eqref{eq:fubini_bound}, by Fubini's theorem,
\begin{align*}
v_x =\int_{0}^{h_n}\int_{0}^{h_n}\frac{1}{st}\E\left[ \Im\{(\widehat{\gamma}_n^\epsilon(s)-\gamma(s))\e^{-is x}\}\Im\{(\widehat{\gamma}_n^\epsilon(t)-\gamma(t))\e^{-it x}\}\right] \diff s \diff t \ ,
\end{align*}
and applying the Cauchy-Schwartz inequality and Theorem~\ref{thm:gamma_s_rate} yields
\begin{equation}\label{eq:vx_rate}
v_x\leq \left( \int_{0}^{h_n}\frac{1}{s}\sqrt{\E|\widehat{\gamma}_n^\epsilon(s)-\gamma(s)|^2} \diff s\right)^2 \leq \frac{C_2}{n}\left( \int_0^{h_n} \frac{\sqrt{s^2}}{s}\diff s\right)^2= \frac{C_2 h_n^2}{n}\ ,
\end{equation}
where the constant $C_2$ is as in Theorem~\ref{thm:gamma_s_rate},
\item \textit{Bound for $b_x^2$:} By Assumption~\ref{assum:G_class}(c),
\begin{equation}\label{eq:bx_rate}
b_x^2 \leq \left( \int_{h_n}^\infty\frac{1}{s} |\gamma(s)|\diff s\right)^2 \leq \left(\int_{h_n}^\infty C_0 s^{-(\eta+1)}\diff s \right)^2= \left(\frac{C_0}{\eta}\right)^2 {h_n}^{-2\eta} \ ,
\end{equation}
where $C_0$ and $\eta$ are as in Assumption~\ref{assum:G_class}(c).
\end{enumerate}
Taking $h_n=n^{1/2(1+\eta)}$ then yields $h_n^2 n^{-1}=h_n^{-2\eta}$. Combing \eqref{eq:vx_rate}-\eqref{eq:bx_rate} with \eqref{eq:risk_up} we conclude that \eqref{eq:risk_rate} holds for some constant $C_3$.
\end{proof}

%%%%%%%%%%%%%%%%%%%%%%%%%%%%%%%%%%%
\section{Proofs and auxiliary results}\label{sec:proof}

This section is devoted to the proofs of Propositions~\ref{prop:cf_cov_rate} and \ref{prop:Z_cov_rate}. Section~\ref{sec:prelim} provides an overview of the properties of the net input and workload distributions. This is followed by a series of lemmas in Section~\ref{sec:corr} that establish the convergence rates of  the correlation functions of the terms appearing in the estimation error, as represented in \eqref{eq:Zi_s}. The proofs of the propositions are then completed in Section~\ref{sec:prop_proofs}.

%%%%%%%%%%%%%%%%%%%%%%%%%%%%%%%%%%
\subsection{Preliminaries}\label{sec:prelim}

In what follows all expectations and probabilities are with respect to the measure induced by a specific job-size distribution $G\in\mathcal{G}$, but for the sake of brevity we omit this from the notation. Denote the traffic intensity of the queue by $\rho=\lambda\E[B]$. It is well known (e.g., \cite{book_A2003}) that $\rho<1$ is a necessary and sufficient condition for positive Harris recurrence of $(V(t))_{t\geq 0}$ and the existence of a stationary distribution $V$ such that $V(t)\darrow V$ as $t\to\infty$. For many identities associated with the distribution of the workload process it is useful to define the characteristic exponent of the LST of the net input process: $\ell(s):=\log\E[\e^{-s X(1)}]$ for $s\in\mathbb{C}$. The stationary workload distribution is given by the generalized Pollaczek–Khinchine (GPK) formula
\begin{equation}\label{eq:GPK}
\E[\e^{-s V}]=\frac{s\P(V=0)}{\ell(s)}=\frac{s(1-\rho)}{\ell(s)}\ , s\in \mathbb{C}\ .
\end{equation}
Observe that for real $s$, taking $s\to\infty$ yields
\begin{equation}\label{eq:P0}
\P(V=0)=1-\rho \ .
\end{equation}
For any $G\in\mathcal{G}$,  $\E[B^2]<\infty$ by Assumption~\ref{assum:G_class}(b), which implies
\begin{equation*}\label{eq:EV}
\E[V]=\frac{\ell^{(2)}(0)}{2\ell\apost(0)}<\infty\ ,
\end{equation*}
where
\begin{equation}\label{eq:ell_p0}
\ell\apost(0):=\lim_{x\downarrow 0}\frac{\diff}{\diff x}\ell(x)=-\E[X(1)]=1-\lambda\E[B]=\P(V=0)\ ,
\end{equation}
and
\begin{equation}\label{eq:ell_p20}
\ell^{(2)}(0):=\lim_{x\downarrow 0}\frac{\diff^2}{\diff x^2}\ell(x)=\Var[ X(1)]=\lambda\E[B^2]\ .
\end{equation}
Similarly,
\begin{equation}\label{eq:ell_p30}
\ell^{(3)}(0):=\lim_{x\downarrow 0}\frac{\diff^3}{\diff x^3}\ell(x)=-\left(\E[X(1)^3]-\E[X(1)]\E[X(1)^2] \right)\ ,
\end{equation}
and $|\ell^{(3)}(0)|<\infty$ if $\E[B^3]<\infty$. For any real $x\geq 0$, denote the inverse of the characteristic exponent by $\psi(x)=\ell^{-1}(x)$, i.e., $\ell(\psi(x))=x$. Standard arguments yield
\begin{equation}\label{eq:psi_p0}
\psi\apost(0):=\lim_{x\downarrow 0}\frac{\diff}{\diff x}\psi(x)=\frac{1}{\ell\apost(0)}\ ,
\end{equation}
\begin{equation}\label{eq:psi_p20}
\psi^{(2)}(0):=\lim_{x\downarrow 0}\frac{\diff^2}{\diff x^2}\psi(x)=-\frac{\ell^{(2)}(0)}{\ell\apost(0)^3}\ ,
\end{equation}
\begin{equation}\label{eq:psi_p30}
\psi^{(3)}(0):=\lim_{x\downarrow 0}\frac{\diff^3}{\diff x^3}\psi(x)=\frac{3\ell^{(2)}(0)}{\ell\apost(0)^5}-\frac{\ell^{(3)}(0)}{\ell\apost(0)^4}\ . 
\end{equation}
Clearly, $\varphi(s)=\ell(-is)$ for any $s\geq 0$, yielding
\begin{equation}\label{eq:phi_p0}
\varphi^{(k)}(0):=\lim_{x\downarrow 0}\frac{\diff^k}{\diff x^k}\varphi(x)=(-i)^k\ell^{(k)}(0)\ , \ k\geq 1 \ .
\end{equation}
Finally as the characteristic function is hermitian,
\begin{equation}\label{eq:phi_h}
\varphi(s)=\overline{\varphi(-s)}\ , \forall s\in\mathbb{R} \ ,
\end{equation}
where $\overline{x}$ denotes the complex conjugate of $x\in\mathbb{C}$.

The workload process is observed according to an independet Poisson sampling process with rate $\xi>0$. Let $(T_j)_{j\geq 1}$ denote the event times of the sampling process and denote the $j$th workload observation by $V_j=V(T_j)$. The distribution of workload observation $j$ conditional on the previous observation was derived in \cite{KBM2006} (see also \cite[Ch.~4.1]{DM2015}), and can be stated as the CF
\begin{equation} \label{eq:lapVT}
\E[\e^{is V_j}|V_{j-1}] = \frac{\xi}{\xi-\varphi(s)} \Big( \e^{is V_{j-1}} + \frac{is}{\xi}\P(V_j=0|V_{j-1})\Big)\ , s\in\mathbb{R}\ ,
\end{equation}
with an atom at zero with probability
\begin{equation}\label{eq:P0_V0}
 \P(V_j=0|V_{j-1})= \frac{\xi \e^{-\psi(\xi)V_{j-1}}}{\psi(\xi)}\ .
\end{equation}
The following Lemma establishes several useful properties of the exponent function and the CF of the stationary workload distribution. 
%%%
\begin{lemma}\label{lemma:phi_bound} If Assumption~\ref{assum:G_class} holds, then the net-input exponent function $\varphi$ satisfies 
\begin{equation}\label{eq:phi_bound}
    \max\{|\Re\{\varphi(s)\}|,|\Im\{\varphi(s)\}|\}\leq |s|(2-\delta)\ , \forall s\in\mathbb{R}\ .
\end{equation}
This further implies that
\begin{equation}\label{eq:phi_s_bound}
\frac{1}{|s|}|\varphi(s)|\leq (2-\delta)\ , \frac{1}{|s|}|\varphi\apost(s)|\leq M_1\ , \forall s\neq 0\ ,
\end{equation}
where $0<M_1<\infty$ is a constant, and that
\begin{equation}\label{eq:rho_bound}
       \lim_{s\to \infty}\E[\e^{isv}]=\P(V=0)\leq (2-\delta)\inf_{s\geq 0} \left|\E[\e^{isv}]\right|.
    \end{equation}
\end{lemma}
\begin{proof} For the sake of brevity the proof focuses on $s>0$, while keeping in mind that the same arguments apply to $s<0$. First consider the derivative of the jump-size CF at $s\geq 0$,
    \begin{align*}
        \gamma'(s)=\E[iB\e^{isB}]=i\E[B(\cos(sB)]-\E[B\sin(sB))]\ ,
    \end{align*}
    which also yields $\gamma'(0)=i\E[B]$. Hence, by  \eqref{eq:phi_LST},
    \begin{align*}
        \varphi'(s)=\lambda\gamma'(s)-i=i(\lambda\E[B(\cos(sB)]-1)-\lambda\E[B\sin(sB))] \ .
    \end{align*}
    As $\lambda\E[B]<1-\delta$, we conclude that for any $ s\geq 0$,
    \begin{align*}
    |\mathbb{E}[B\sin(sB)]|< 1-\delta \ , \ |\lambda\mathbb{E}[B\cos(sB)]-1| < 2-\delta\ , 
\end{align*}
hence
    \begin{align*}
    \max\{|\Re\{\varphi'(s)\}|,|\Im\{\varphi'(s)\}|\}=\max\{|\lambda\mathbb{E}B\cos(sB)-1|,\lambda|\mathbb{E}B\sin(sB)|\}<2-\delta\ .
\end{align*}
This further implies \eqref{eq:phi_bound} because $\varphi(s)=\int_0^s\varphi'(u)\diff u$ (recalling that $\varphi(0)=0$). By \eqref{eq:phi_p0} and Assumption~\ref{assum:G_class}(a) we have that $|\varphi'(0)|<1-\delta$, and applying L'H\^opital's rule yields 
\begin{align*}
    \lim_{s\downarrow 0}\left|\frac{\varphi(s)}{s}\right|=|\varphi'(0)|<1-\delta<2-\delta\ .
\end{align*}
Hence, as $|\varphi'(s)|<1-\delta$ for any $s>0$, we have that
\begin{align*}
    \left|\frac{\varphi(s)}{s}\right|=\frac{|\int_0^s\varphi'(u)\diff u|}{|s|}\leq 2-\delta\ .
\end{align*}
Similarly, by Assumption~\ref{assum:G_class}(b),
\begin{align*}
    \lim_{s\downarrow 0}\left|\frac{\varphi'(s)}{s}\right|=  \left|\varphi^{(2)}(0)\right|<\infty \ .
\end{align*}
Further observe that $|\varphi'(s)|$ is a continuous function such that $|\varphi'(0)|<1$ and $\lim_{s\to\infty}|\varphi'(s)|=0$. This implies that there exists some $s_0>0$ such that $|\varphi'(s)|/s\leq m_1$ for all $s> s_0$. We conclude that  the second part of \eqref{eq:phi_s_bound} holds with $M_1=\max\{m_1,\sup_{s\in[0,s_0]}|\varphi'(s)|/s\}$.
Finally, the random variable $V$ has an atom on the event $\{V=0\}$ and is absolutely continuous on $(0,\infty)$ with a density $f_V$. The Riemann–Lebesgue Lemma then implies that
    \begin{equation*}
       \E[\e^{isv}]=\P(V=0)+\P(V>0)\int_0^\infty \e^{is u}f_V(u)\diff u\overset{s\to\infty}{\rightarrow} \P(V=0).
    \end{equation*}
    Therefore, by \eqref{eq:phi_s_bound} and the GPK formula \eqref{eq:GPK}, for any $s\geq 0$,
    \begin{align*}
      \left|  \E[\e^{isV]}\right|=\frac{\P(V=0)}{\left|\frac{\varphi(s)}{s}\right|}\geq \frac{\P(V=0)}{2-\delta}\ .
    \end{align*}
\end{proof}

\subsection{Convergence of the covariance terms}\label{sec:corr}

We are interested in the convergence rate of covariance terms of the form $\Cov(g_s(V_0),g_s(V_n))$, where $g_s(\cdot)$ is a bounded function indexed by a complex parameter $s$. As a preliminary step, Lemma~\ref{lemma:quad_ergodic} is a general convergence rate result for the workload process sampled according to a Poisson process. This relies on the quadratic ergodicity of the M/G/1 workload process with a finite third moment of the job-size distribution (see \cite{TT1994} and \cite[Ch.~14.4]{book_MT2009}). The role of this lemma is to facilitate the application of Fubini's theorem in evaluating the expectation of infinite sums in the subsequent analysis.

%%%
\begin{lemma}\label{lemma:quad_ergodic}
Under Assumptions~ \ref{assum:G_class} and \ref{assum:stationary}, let \( g: [0, \infty) \to \mathbb{C} \) be a function satisfying \( |g(v)| < M_3 \) for all \( v \geq 0 \), where \( 0 < M_3 < \infty \). Then, with probability one,  
\begin{align*}  
\sum_{k=1}^\infty \left| \mathbb{E}[g(V_k) \mid V_0] - \mathbb{E}[g(V)] \right| < \infty.  
\end{align*}  
\end{lemma}
\begin{proof}
First note that $\P(V_0<\infty)=1$ when $V_0$ follows the stationary distribution. For an M/G/1 process, it is known that if $\lambda\E[B]<1$ and $\E[B^{k+1}]<\infty$, then the workload process at arrival moments is ergodic with a convergence rate of $n^k$ (see \cite{TT1994}). We will show that this result holds for the workload process sampled according to a Poisson process. To this end we first construct a discrete-time version of the workload process at arrival and sampling times. Due to the memoryless property, at any state of the system the time until the next event of an arrival or an observation is an exponential random variable with rate $\lambda+\xi$. Let  $(\tau_k)_{k\geq 0}$ denote an iid sequence of such inter-event times. Let $I_k$ denote the indicator of event $k$ being a an arrival event, then $I_k$ is a Bernoulli random variable with probability $\lambda/(\lambda+\xi)$, and it is also independent of $\tau_k$. Let $\tilde{V}_k:=\lim_{t\downarrow 0}V(\sum_{j=1}^k\tau_j-t)$ denote the workload just before the $k$'th event. The standard Lindley recursion yields
\begin{align*}
    \tilde{V}_{n}=\max\left\lbrace \tilde{V}_{n-1}+B_{n-1}I_{n-1}-\tau_n,0\right\rbrace\ , n\geq 1\ ,
\end{align*}
with initial values $\tilde{V}_0=V_0$ and $B_0=I_0=0$. Note that in this construction there are job sizes $B_n$ that are generated for sampling events even though they do not really correspond to work added to the system. The process $(\tilde{V}_n)_{n\geq 0}$ is a reflected random walk with increments $U_n=B_{n-1}I_{n-1}-\tau_n$. By Assumption~\ref{assum:G_class}(a), the mean increment satisfies
\begin{align*}
    \E[U_1]=\frac{\lambda\E[B]}{\lambda+\xi}-\frac{1}{\lambda+\xi}<-\frac{\delta}{\lambda+\xi} <0\ , \forall G\in\mathcal{G} \ .
\end{align*}
Moreover, Assumption~\ref{assum:G_class}(b) implies that there exists a constant $u_1<\infty$ such that  $\E[U_1^3]<u_1$, $\forall G\in\mathcal{G}$. Then, by \cite[Thm.~4.1 and Prop.~5.1]{TT1994}, 
\begin{equation}\label{eq:poly_g}
\lim_{n \to\infty}n^2|\E[g(\tilde{V}_n)|V_0] -E[g(V)]|= 0 \ .
\end{equation}
for any bounded function $g$ and initial state $V_0\in[0,\infty)$. Let $S(k)$ denote the index of the $k$'th sampling event in the sequence of events driving $(\tilde{V})_{n\geq 1}$, and note that $S(k)-S(k-1)$ follows a geometric distribution with parameter $\xi/(\lambda+\xi)$. By the SLLN, $S(k)/k\asarrow \xi/(\lambda+\xi)$ as $k\to\infty$. Clearly, $S(k)\geq k$, and in particular $S(k)\asarrow \infty$ as $k\to\infty$. Thus, \eqref{eq:poly_g} implies that
\begin{align*}
\lim_{k\to\infty} k^2|\E[g(V_k)|V_0]-\E[g(V)]|=\lim_{k\to\infty}  \left(\frac{k}{S(k)}\right)^2 S(k)^2|\E[g(\tilde{V}_{S(k)})|V_0]-\E[g(V)]| =0\ ,
\end{align*}
almost surely for any $V_0<\infty$.
We conclude that
\begin{align*}
\sum_{k=1}^\infty |\E[g(V_k)|V_0]-\E[g(V)]|<\infty \ .
\end{align*}
\end{proof}

We next derive an explicit expression for the cumulative deviation of sequences of expectations of certain functions of the workload from their stationary counterpart, given an initial workload. The proofs rely on arguments introduced in \cite{MR2021} and applying Lemma~\ref{lemma:quad_ergodic}. Note that the results pertaining to the workload LST are stated for real values in \cite{MR2021}, i.e., $\E[\e^{-sV}]$ for $s\in[0,\infty)$. However, the proofs do not rely on this assumption and carry over directly to the case of complex values.

%%%
\begin{lemma}\label{lemma:conv_gV}
Assume that $V_0:=V(0)<\infty$ almost surely. Then, for any $s\in\mathbb{R}$ and $G\in\mathcal{G}$ as defined in Assumption~\ref{assum:G_class},
\begin{align*}
\sum_{k=1}^\infty\left(\E[\e^{-is V_k}|V_0]-]\E[\e^{-is V_0}]\right)=\xi k_1(V_0,s)\ ,
\end{align*}
where 
\begin{align*}
k_1(v,s)=\frac{1}{2\varphi(s)}\left(2\e^{is v}\left(isv-1\right)+\frac{s\varphi\apost(0)}{\varphi(s)}\left(2+\frac{\varphi(s)\ell^{(2)}(0)}{\ell\apost(0)^2}\right)\right)\ , v\geq 0,s\in\mathbb{R} \ .
\end{align*}
\end{lemma}
\begin{proof}
Let $Y_q$ denote an exponential random variable with rate $q>0$.
 By \cite[Lemma.~5]{MR2021} and the dominated convergence theorem, for any $V_0<\infty$,
\begin{align*}
\sum_{k=1}^\infty\left(\E[\e^{isV_k}|V_0]-\E[\e^{is V}]\right) &= \xi\int_0^\infty \left(\E[\e^{isV(u)}|V_0]-\E[\e^{is V}]\right)\diff u \\
&=
\xi \lim_{q\downarrow 0}\frac{1}{q}\int_0^\infty q \e^{-q u} \left(\E[\e^{isV(u)}|V_0]-\E[\e^{is V}]\right)\diff u \\
&=
\xi \lim_{q\downarrow 0}\frac{1}{q}\left( \E[\e^{isV(Y_q)}|V_0]-\E[\e^{is V_0}]\right)
 \ .
\end{align*}
Applying \eqref{eq:lapVT} and the GPK formula \eqref{eq:GPK} yields 
\begin{align*}
\sum_{k=1}^\infty\left(\E[\e^{isV_k}|V_0]-\E[\e^{is V_0}]\right) &=
\xi \lim_{q\downarrow 0}\frac{1}{q}\left(\frac{q}{q-\varphi(s)} \Big( \e^{is V_0} + \frac{is \e^{-\psi(q)V_0}}{\psi(q)}\Big)-\frac{s\varphi\apost(0)}{\varphi(s)}\right) \\
&=\xi \lim_{q\downarrow 0}\frac{q\left(\psi(q)\e^{isV_0}+is\e^{-\psi(q)V_0}\right)-\frac{s\varphi\apost(0)}{\varphi(s)}\psi(q)(q-\varphi(s))}{q(q-\varphi(s))\psi(q)}
 \ .
\end{align*}
Applying L'H\^opital's rule twice yields
\small{
\begin{align*}
& \lim_{q\downarrow 0}\frac{q\left(\psi(q)\e^{isV_0}+is\e^{-\psi(q)V_0}\right)-\frac{s\varphi\apost(0)}{\varphi(s)}\psi(q)(q-\varphi(s))}{q(q-\varphi(s))\psi(q)} \\
&= \lim_{q\downarrow 0} \frac{\psi(q)\e^{isV_0}+is\e^{-\psi(q)V_0}+q\left(\psi\apost(q)\e^{isV_0}-is\psi(q)V_0\e^{-\psi(q)V_0}\right)-\frac{s\varphi\apost(0)}{\varphi(s)}(\psi\apost(q)(q-\varphi(s))+\psi(q))}{\psi(q)(q-\varphi(s))+q(\psi(q)+\psi\apost(q)(q-\varphi(s))} \\
&= \lim_{q\downarrow 0} \frac{2\psi\apost(q)\e^{isV_0}(1-isV_0)+q\left(\psi^{(2)}(q)\e^{isV_0}+is(\psi(q)V_0)^2\e^{-\psi(q)V_0}\right)-\frac{s\varphi\apost(0)}{\varphi(s)}(\psi^{(2)}(q)(q-\varphi(s))+2\psi\apost(q))}{2\psi(q)+2\psi\apost(q)(q-\varphi(s))+q(2\psi\apost(q)-\varphi(s))\psi^{(2)}(q))} \\
&= \frac{2\psi\apost(0)\e^{isV_0}(1-isV_0)-\frac{s\varphi\apost(0)}{\varphi(s)}(2\psi\apost(0)-\psi^{(2)}(0)\varphi(s))}{-2\varphi(s)\psi\apost(0)} 
 \ .
\end{align*}}
The proof is completed by applying identities \eqref{eq:psi_p0}-\eqref{eq:psi_p20}, keeping in mind that $\E[B^3]<\infty$ together with \eqref{eq:ell_p0}-\eqref{eq:ell_p30} imply that all terms in $k_1(\cdot)$ are finite.
\end{proof}

%%%
\begin{lemma}\label{lemma:conv_kV}
Suppose Assumptions~\ref{assum:G_class}-\ref{assum:stationary} hold. Then, for any $(s_1,s_2)\in\mathbb{R}^2$ and $G\in\mathcal{G}$,
\begin{align*}
\sum_{k=1}^\infty\left(\E[\e^{i(s_1V_0+s_2V_k)}]-\E[\e^{is_1 V}]\E[\e^{is_2 V}]\right)=\xi \kappa_1(s_1,s_2;G)\ ,
\end{align*}
where
\begin{align*}
\kappa_1(s_1,s_2;G)=\frac{s_1\ell\apost(0)}{\varphi(s_1)}\left(\frac{\varphi\apost(0)(s_2\varphi(s_2)-s_2^2\varphi\apost(s))}{\varphi(s_2)^3}+ \frac{s_2\varphi\apost(0)\ell^{(2)}(0)}{2\varphi(s_2)\ell\apost(0)^2}\right) \ .
\end{align*}
Moreover, there exists a constant $K_1$ such that
\begin{align*}
\sup_{G\in\mathcal{G}}\sup_{(s_1,s_2)\in\mathbb{R}^2}|\kappa_1(s_1,s_2;G)|\leq K_1<\infty \ .
\end{align*}
\end{lemma}
\begin{proof}[Proof of Lemma~\ref{lemma:conv_kV}]
We make use of the following identity (see \cite[Eq.~(11)]{MR2021}),
\begin{align*}
\E[iV_0\e^{is V_0}]=\frac{\varphi\apost(0)(\varphi(s)-s\varphi\apost(s))}{\varphi(s)^2}\ .
\end{align*}
Then by Lemma~\ref{lemma:conv_gV}, applying \eqref{eq:GPK}, \eqref{eq:psi_p0}, and \eqref{eq:psi_p20}
yields 
\begin{align*}
\E[k_1(V_0,s)] &= \frac{1}{2\varphi(s)}\left(\frac{2s\varphi\apost(0)(\varphi(s)-s\varphi\apost(s))}{\varphi(s)^2}-\frac{2s\varphi\apost(0)}{\varphi(s)}+\frac{s\varphi\apost(0)}{\varphi(s)}\left(2+\frac{\varphi(s)\ell^{(2)}(0)}{\ell\apost(0)^2}\right)\right) \\
&=\frac{\varphi\apost(0)(s\varphi(s)-s^2\varphi\apost(s))}{\varphi(s)^3}+ \frac{s\varphi\apost(0)\ell^{(2)}(0)}{2\varphi(s)\ell\apost(0)^2} \ .
\end{align*}
Applying Lemma~\ref{lemma:quad_ergodic} and Fubini's theorem we have that
\begin{align*}
\sum_{k=1}^\infty\left(\E[\e^{i(s_1V_0+s_2V_k)}]-\E[\e^{is_1 V}]\E[\e^{is_2 V}]\right) &= \sum_{k=1}^\infty\left(\E[\E[\e^{-i(s_1V_0+s_2V_k)}|V_0]]-\E[\e^{is_1 V}]\E[\e^{is_2 V}]\right) \\
&= \sum_{k=1}^\infty\E[\e^{-is_1 V}]\left(\E[\E[\e^{-is_2V_k}|V_0]]-\E[\e^{-is_2 V}]\right) \\
&=\E[\e^{-is_1 V}]\E\left[ \sum_{k=1}^\infty\left(\E[\e^{-is_1V_k}|V_0]-\E[e^{-is_2 V}]\right)\right] \\
&= \xi \E[k_1(V_0,s_2)]\E[\e^{-is_1 V}]=\xi\kappa_1(s;G) \ .
\end{align*}
Applying L'H\^opital's rule three times to $\kappa_1(s_1,s_2;G)$ to establish that the limit as $s_2\to 0$ is bounded under the following condition,
\begin{align*}
\sup_{G\in\mathcal{G}}|\kappa_1(0,0;G)|<\infty \ \Leftrightarrow \ \sup_{G\in\mathcal{G}}\varphi^{(3)}(0)<\infty\ . 
\end{align*}
This condition is satisfied by \eqref{eq:ell_p30} and Assumption~\ref{assum:G_class}(b). Finally, Lemma~\ref{lemma:phi_bound} ensures that as $\kappa_1(s_1,s_2;G)$ is continuous (with respect to $s_1$ and $s_2$) and uniformly bounded (with respect to $G\in\mathcal{G}$) and the statement follows.
\end{proof}

%%%
\begin{lemma}\label{lemma:conv_Vk=0}
Assume that $V_0:=V(0)<\infty$ almost surely. Then, for any $G\in\mathcal{G}$ as defined in Assumption~\ref{assum:G_class},
\begin{align*}
k_2(V_0):=\sum_{k=1}^\infty\left(\P(V_k=0|V_0)-\P(V=0)\right)=\frac{\xi}{2} \left(\frac{\ell^{(2)}(0)}{\ell\apost(0)}-2V_0\right)\ .
\end{align*}
\end{lemma}
\begin{proof}
Applying \cite[Lemma.~5]{MR2021} yields
\begin{align*}
\sum_{k=1}^\infty\left(\P(V_k=0|V_0=v)-\P(V=0)\right) = \xi\lim_{q\downarrow 0}\frac{1}{q}\left(\P(V(Y_q)=0|V_0=v)-\P(V=0)\right) \ .
\end{align*}
Plugging in \eqref{eq:P0} and \eqref{eq:P0_V0}, and then computing the limit via L'H\^opital's rule yields
\begin{align*}
k_2(v) &=\xi\lim_{q\downarrow 0}\frac{q\e^{-\psi(q)v}-\psi(q)\P(V=0)}{q\psi(q)} \\
&= \xi\lim_{q\downarrow 0} \frac{\e^{-\psi(q)v}-q\psi\apost(q)v\e^{-\psi(q)v}-\psi\apost(q)\ell\apost(0)}{\psi(q)+q\psi\apost(q)}\\
&= \xi\lim_{q\downarrow 0}\frac{-2\psi\apost(q)v\e^{-\psi(q)v}-q(\psi^{(2)}(q)v\e^{-\psi(q)v}-(\psi^{'}(q)v)^2\e^{-\psi(q)v})-\psi^{(2)}(q)\ell\apost(0)}{2\psi^{'}(q)+q\psi^{(2)}(q)} \\
&= \xi\frac{-2\psi\apost(0)v-\psi^{(2)}(0)\ell\apost(0)}{2\psi\apost(0)} \ .
\end{align*}
The result follows from \eqref{eq:psi_p0} and \eqref{eq:psi_p20}.
\end{proof}

%%%
\begin{lemma}\label{lemma:conv_Vk0=0}
Suppose Assumption~\ref{assum:G_class} holds. Then, for any $G\in\mathcal{G}$,
\begin{align*}
\sum_{k=1}^\infty\left(\P(V_k=V_0=0)-\P(V=0)^2\right)=\frac{\xi\ell^{(2)}(0)}{2}<\infty\ .
\end{align*}
\end{lemma}
\begin{proof}
The conclusion that the series converges is a direct corollary of Lemma~\ref{lemma:quad_ergodic} for $g(v)=\mathbf{1}\{v=0\}$. To obtain the explicit term for the series we apply \cite[Lemma.~5]{MR2021} once more. By Lemma~\ref{lemma:conv_Vk=0}, as $V_0\sim V$, 
\begin{align*}
\sum_{k=1}^\infty\left(\P(V_k=V_0=0)-\P(V=0)^2\right) &= \sum_{k=1}^\infty\left(\P(V_k=0|V_0=0)\P(V_0=0)-\P(V=0)^2\right) \\
&= \P(V=0)\E[k_2(0)]= \frac{\P(V=0)\xi\ell^{(2)}(0)}{2\ell\apost(0)}\ .
\end{align*}
Finally, plugging in $\P(V=0)= \ell\apost(0)$ (see Eq.~\ref{eq:P0}) and identities \eqref{eq:psi_p0}-\eqref{eq:psi_p20} yield the result.
\end{proof}

%%%
\begin{lemma}\label{lemma:conv_V0eis}
Suppose Assumptions~\ref{assum:G_class}-\ref{assum:stationary} hold. Then, for any $s\in\mathbb{R}$ and $G\in\mathcal{G}$,
\begin{align*}
\kappa_2(s;G)=\sum_{k=1}^\infty\left(\E[\mathbf{1}\{V_k=0\}\e^{-is V_0}]-\P(V=0)\E[\e^{-is V}]\right)=\frac{\xi \varphi\apost(0)\ell^{(2)}(0)s}{2\ell\apost(0)\varphi(s)}\ .
\end{align*}
Moreover, there exists a constant $K_2$ such that 
\begin{align*}
\sup_{G\in\mathcal{G}}\sup_{s\geq 0}|\kappa_2(s;G)|\leq K_2<\infty\ .
\end{align*}
\end{lemma}
\begin{proof}
The proof follows similar arguments to as in the previous proofs. In particular, Lemmas~\ref{lemma:quad_ergodic} and~\ref{lemma:conv_Vk=0} yield
\begin{align*}
\kappa_2(s;G)&=\sum_{k=1}^\infty \left(\E[\e^{isV_0}\mathbf{1}\{V_k=0\}]-\E[\e^{is V}]\P(V=0)\right) = \E[\e^{is V}]\E\left[\sum_{k=1}^\infty \left(\P(V_k=0|V_0)-\P(V=0)\right)\right] \\
&= \E[\e^{is V}]\E[k_2(0)] =\frac{\xi \varphi\apost(0)\ell^{(2)}(0)s}{2\ell\apost(0)\varphi(s)} \ .
\end{align*}
Lemma~\ref{lemma:phi_bound} established that $\sup_{G\in\mathcal{G}}\sup_{s\geq 0}|\varphi(s)|/s<\infty$,  and thus $\sup_{G\in\mathcal{G}}\sup_{s\geq 0}|\kappa_2(s;G)|<\infty$ as well.
\end{proof}

%%%
\begin{lemma}\label{lemma:conv_Vkeis}
Suppose Assumptions~\ref{assum:G_class}-\ref{assum:stationary} hold. Then, for any $s\in\mathbb{R}$ and $G\in\mathcal{G}$,
\begin{align*}
\kappa_3(s;G) &=\sum_{k=1}^\infty\left(\E[\mathbf{1}\{V_0=0\}\e^{-is V_k}]-\P(V=0)\E[\e^{-is V}]\right) \\
&=\xi\ell'(0)\left(\frac{s\varphi\apost(0)}{\varphi(s)^2} -\frac{s^2\varphi\apost(s)\varphi\apost(0)}{\varphi(s)^3}+ \frac{s\varphi\apost(0)\ell^{(2)}(0)}{2\varphi(s)\ell\apost(0)^2}\right)\ .
\end{align*}
Moreover, there exists a constant $K_3$ such that 
\begin{align*}
\sup_{G\in\mathcal{G}}\sup_{s\geq 0}|\kappa_3(s;G)|\leq K_3<\infty \ .
\end{align*}
\end{lemma}
\begin{proof}
As in the previous proofs, applying Lemmas~\ref{lemma:quad_ergodic} and~\ref{lemma:conv_gV} yields
\begin{align*}
\kappa_3(s;G)&=\sum_{k=1}^\infty \left(\E[\mathbf{1}\{V_0=0\}\e^{-is V_k}]-\P(V=0)\E[\e^{-is V}]\right) \\
&= \P(V=0)\E\left[\sum_{k=1}^\infty \left(\E[\e^{-is V_k}|V_0]-\E[\e^{-is V}]\right)\right] \\
&= \xi\P(V=0)\E[k_1(V_0)] =\xi\ell'(0)\left(\frac{s\varphi\apost(0)}{\varphi(s)^2} -\frac{s^2\varphi\apost(s)\varphi\apost(0)}{\varphi(s)^3}+ \frac{s\varphi\apost(0)\ell^{(2)}(0)}{2\varphi(s)\ell\apost(0)^2}\right) \ .
\end{align*}
Finally, repeating the arguments in the proof of Lemma~\ref{lemma:conv_kV} and applying Lemma~\ref{lemma:phi_bound} yields that $\sup_{s\geq 0}|\kappa_3(s;G)|$ is uniformly bounded on $\mathcal{G}$.
\end{proof}

%%%%%%%%%%%%%%%%%%%%%%%%%%%%%%%%%%
\subsection{Proofs of Propositions~\ref{prop:cf_cov_rate} and~\ref{prop:Z_cov_rate}}\label{sec:prop_proofs}

%%%
\begin{proof}[Proof of Proposition~\ref{prop:cf_cov_rate}]
Assume $G\in\mathcal{G}$, then $\lim_{s\to\infty}|\E[\e^{is V}]|=\P(V=0)>\delta$ by Assumption~\ref{assum:G_class}(a). Combining this with \eqref{eq:rho_bound} in Lemma~\ref{lemma:phi_bound} we have that for any $\delta\in(0,\frac{1}{2})$,
\begin{align*}
    \epsilon< \frac{\delta}{4}<\frac{\P(V=0)}{2(2-\delta)} \leq\frac{\inf_{s\geq 0}|\E[\e^{is V}]|}{2}\ . 
\end{align*}
Hence, 
\begin{align*}
\left\lbrace \left|\frac{1}{n}\sum_{j=1}^n\e^{is V_j}\right|\leq \epsilon \right\rbrace \subseteq \left\lbrace \left|\frac{1}{n}\sum_{j=1}^n\e^{-is V_j}-\E[\e^{is V}]\right|>\epsilon\right\rbrace \ .
\end{align*}
By Chebyshev's inequality,
\begin{align*}
\P\left(\left|\frac{1}{n}\sum_{j=1}^n\e^{is V_j}\right|\leq \epsilon \right)\leq \P \left(\left|\frac{1}{n}\sum_{j=1}^n\e^{is V_j}-\E[\e^{is V}]\right|>\epsilon\right)\leq \frac{\Var\left[\frac{1}{n}\sum_{j=1}^n\e^{is V_j}\right]}{\epsilon^2} \ .
\end{align*}
Assumption~\ref{assum:stationary} implies that $\Cov(g(V_j),g(V_{j+k}))=\Cov(g(V_0),g(V_{k}))$ for any function $g$ and $j,k\geq 0$, hence
\begin{align*}
\Var\left[\frac{1}{n}\sum_{j=1}^n\e^{is V_j}\right]&= \frac{1}{n^2}\sum_{j=1}^n\sum_{k=1}^n \Cov(\e^{is V_j},\e^{is V_k}) \\
&= \frac{1}{n}\left(\Var[\e^{is V}]+\sum_{k=1}^{n-1}\frac{n-k}{n}\Cov(\e^{is V_0},\e^{is V_k})+\sum_{k=1}^{n-1}\frac{n-k}{n}\Cov(\e^{is V_k},\e^{is V_0})\right)\ .
\end{align*}
Clearly, $\Var[\e^{is V}]$ is uniformly bounded in $\mathcal{G}$ by a constant (for any $s\geq 0$). As the random variables in the covariance series are complex valued and the CF is hermitian, we have that
\begin{align*}
\Cov(\e^{is V_0},\e^{is V_k}) &= \E[\e^{-is V_0+is V_k}]-\E[\e^{-is V_0}]\E[\e^{is V_0}] \ , \\
\Cov(\e^{is V_k},\e^{is V_0}) &= \E[\e^{is V_0-is V_k}]-\E[\e^{is V_0}]\E[\e^{-is V_0}] \ .
\end{align*}
Applying Lemma~\ref{lemma:conv_kV}, once with $s_1=-s$ and $s_2=s$, and again with $s_1=s$ and $s_2=-s$, under Assumption~\ref{assum:G_class}, there exists a constant $K_1$ such that 
\begin{align*}
\left|\sum_{k=1}^\infty \left(\E[\e^{-is V_0+is V_k}]-\E[\e^{-is V_0}]\E[\e^{is V_0}]\right)\right| \leq K_1 <\infty\ , \forall s\geq 0\ .
\end{align*}
Then, for any $d>0$,
\begin{align*}
   -K_1-d<\max\left\lbrace\sum_{k=1}^\infty\Re\left\lbrace\Cov(\e^{is V_0},\e^{is V_k})\right\rbrace,\sum_{k=1}^\infty\Im\left\lbrace\Cov(\e^{is V_0},\e^{is V_k})\right\rbrace\right\rbrace<K_1+d \ ,
\end{align*}
which implies that
\begin{align*}
    \lim_{n\to\infty}\frac{1}{n}\sum_{k=1}^{n-1}\Cov(\e^{is V_0},\e^{is V_k})=0 \ .
\end{align*}
Observe that for $n\geq 2$,
\begin{align*}
    \left|\sum_{k=1}^{n-1}\frac{k}{n}\Cov(\e^{is V_0},\e^{is V_k})\right|\leq  \left|\sum_{k=1}^{n-1}\Cov(\e^{is V_0},\e^{is V_k})\right| \ ,
\end{align*}
hence we also have that
\begin{align*}
    \lim_{n\to\infty}\frac{1}{n}\sum_{k=1}^{n-1}\frac{k}{n}\Cov(\e^{is V_0},\e^{is V_k})=0 \ .
\end{align*}
As the bound in Lemma~\ref{lemma:conv_kV} is uniform on $\mathcal{G}$, we conclude that there exists a constant $0<C_1<\infty$ such that
\begin{align*}
\sup_{G\in\mathcal{G}}\Var\left[\frac{1}{n}\sum_{j=1}^n\e^{is V_j}\right]\leq \frac{C_1}{n} \ .
\end{align*}
\end{proof}

%%%
\begin{proof}[Proof of Proposition~\ref{prop:Z_cov_rate}]
By construction of the estimator, we have that $\E[Z_j(s)]=0$ for any $j\geq 1$. Furthermore, Assumption~\ref{assum:stationary} implies that $Z_1(s)$ is stationary, and thus
\begin{align*}
\E|\frac{1}{n}\sum_{j=1}^nZ_j(s)|^2 &= \Var\left[\frac{1}{n}\sum_{j=1}^nZ_j(s)\right] \\
&=\frac{1}{n^2}\left(n\Var[Z_1(s)]+\sum_{k=2}^n(n-k)\Cov(Z_1(s),Z_k(s))+\sum_{k=2}^{n}(n-k)\Cov(Z_k(s),Z_1(s))\right)\\
&=\frac{1}{n}\left(\Var[Z_1(s)]+\sum_{k=2}^n\frac{n-k}{n}\E[Z_1(s)\overline{Z_k(s)}]+\sum_{k=2}^n\frac{n-k}{n}\E[Z_k(s)\overline{Z_1(s)}]\right)\
\ .
\end{align*}
We next construct a bound for the first series, i.e., the sum of $\E[Z_1(s)\overline{Z_k(s)}]$. An identical argument establishes the same bound for the last series with the term of $\E[Z_k(s)\overline{Z_1(s)}]$.

For $j\geq 1$ and $s\geq 0$, let $R_j(s)=\Re\{Z_j(s)\}$ and $I_j(s)=-\Im\{Z_j(s)\}$, i.e., $Z_j(s)=R_j(s)+iI_j(s)$. Let $\varphi(s)=u_s+iv_s$, then we can write
\begin{align*}
R_j(s) &= (\xi-u_s)\cos(s V_j)+v_s\sin(s V_j)-\xi\cos(sV_{j-1}) \ , \\ 
I_j(s) &= (\xi-u_s)\sin(s V_j)-v_s\cos(s V_j)-\xi\sin(sV_{j-1})-s\mathbf{1}\{V_j=0)\} \ .
\end{align*}
Straightforward manipulations yield
\begin{align*}
\overline{Z_j(s)}&= R_j(s)-iI_j(s)= (\xi-u_s-v_s)(\cos(sV_j)-i\sin(sV_j))-\xi(\cos(sV_{j-1}+i\sin(sV_{j-1}))+is\mathbf{1}\{V_j=0\} \\
&= (\xi-\varphi(-s))\e^{-is V_j}-\xi\e^{-is V_{j-1}}+is\mathbf{1}\{V_j=0\} \ ,
\end{align*}
where the last equality used \eqref{eq:phi_h}, i.e., the hermitian property of $\varphi$.
Combining the above yields 
\begin{align*}
Z_1(s)\overline{Z_k(s)}&= (\xi-\varphi(s))(\xi-\varphi(-s))\e^{-is(V_k-V_1)}+\xi^2\e^{-is(V_{k-1}-V_{0})}+s^2\mathbf{1}\{V_1=V_k=0\} \\ 
& -\xi(\xi-\varphi(s))\e^{-is(V_{k-1}-V_{0})}+is(\xi-\varphi(s))\e^{isV_1}\mathbf{1}\{V_k=0\} -\xi(\xi-\varphi(s))\e^{-is(V_k-V_{0})}\\
& - is(\xi-\varphi(-s))\e^{-isV_k}\mathbf{1}\{V_1=0\} -\xi is\e^{is V_0}\mathbf{1}\{V_k=0\}-\xi is\e^{-is V_{k-1}}\mathbf{1}\{V_1=0\} \\
&= A_k(s)+B_k(s)+C_k(s)+D_k(s) \ ,
\end{align*}
where 
\begin{align*}
A_k(s) &= (\xi-\varphi(s))(\xi-\varphi(-s))\e^{-is(V_k-V_1)}+\xi\varphi(s)\e^{-is(V_{k-1}-V_{0})}-\xi(\xi-\varphi(s))\e^{-is(V_k-V_{0})} \ , \\
B_k(s) &= is\left[(\xi-\varphi(s))\e^{isV_1}\mathbf{1}\{V_k=0\} -\xi \e^{is V_0}\mathbf{1}\{V_k=0\}\right] \ , \\
C_k(s) &=-is\left[(\xi-\varphi(-s))\e^{-isV_k}\mathbf{1}\{V_1=0\} +\xi\e^{-is V_{k-1}}\mathbf{1}\{V_1=0\}\right] \ , \\
D_k(s) &= s^2\mathbf{1}\{V_1=V_k=0\} \ .
\end{align*}
Lemma~\ref{lemma:conv_kV} implies that there exists a constant $K_4<\infty$ such that 
\begin{align*}
\left|\sum_{k=1}^\infty \E[A_k(s)]-\E[A_1(s)]\right|\leq K_4\ ,
\end{align*}
where we $A_1(s)$ is the stationary mean due to Assumption~\ref{assum:stationary}. Similarly, Lemmas~\ref{lemma:conv_V0eis}, \ref{lemma:conv_Vkeis} and \ref{lemma:conv_Vk0=0}, respectively imply that there exist constants $K_5,K_6,K_7$ such that
\begin{align*}
\left|\sum_{k=1}^\infty \E[B_k(s)]-\E[B_1(s)]\right| &\leq s K_5<\infty\ , \\
\left|\sum_{k=1}^\infty \E[C_k(s)]-\E[C_1(s)]\right| &\leq s K_6<\infty\ , \\
\left|\sum_{k=1}^\infty \E[D_k(s)]-\E[D_1(s)]\right| &\leq s^2 K_7<\infty\ .
\end{align*} 
Let $\bar{K}=\max\{K_4,K_5,K_6,K_7\}$. As $\E[Z_1(s)]=0$, combining the above with the triangle inequality yields
\begin{align*}
\left|\sum_{k=1}^n\E[Z_1(s)\overline{Z_k(s)}]\right|=\left|\sum_{k=1}^n\left(\E[Z_1(s)\overline{Z_k(s)}]-\E^2[Z_1(s)]\right) \right|\leq \max\{1,s^2\}\bar{K} \ .
\end{align*}
By the same argument as in the proof of Proposition~\ref{prop:cf_cov_rate} we conclude
\begin{align*}
\left|\sum_{k=2}^n\frac{n-k}{n}\E[Z_1(s)\overline{Z_k(s)}] \right|\leq 2\max\{1,s^2\}\bar{K}\ .
\end{align*}
For the lower bound we can focus on the variance term
\begin{align*}
\Var[Z_1(s)]=\E[Z_1(s)\overline{Z_1(s)}]\ ,
\end{align*} 
because if it has a lower bound of order $s^2/n$ then the order of the covariance term is not relevant (due to the upper bound established above). The exponential functions and indicators in the above expression for $Z_1(s)\overline{Z_1(s)}$ are all bounded by one in absolute value, hence $\E[Z_1(s)\overline{Z_1(s)}]$ is bounded from above and below by quadratic functions of $s\geq 1$. Thus, the proof is complete.
\end{proof}

%%%%%%%%%%%%%%%%%
\section{Unknown arrival rate}\label{sec:lambda}

So far we have assumed that the arrival rate of jobs, $\lambda$, is known. However, this may not be the case in various applications. Therefore, we would like to be able to simultaneously estimate both elements of the input process, i.e., $G$ and $\lambda$. One way to do this is by utilizing the stationary expected utilization rate:
\begin{align*}
\P(V>0)=\rho=\lambda\E[B]=\lambda\int_0^\infty (1-G(x))\diff x\ .
\end{align*}
Plugging in the inversion formula \eqref{eq:G_inv} we have,
\begin{align*}
\P(V>0) = \lambda\int_0^\infty \left(\frac{1}{\pi}\int_0^\infty\frac{1}{s} \Im\{\gamma(s)\e^{-i s x}\}\diff s-\frac{1}{2}\right)\diff x \ .
\end{align*}
The estimators $(\widehat{\lambda}_n,\widehat{G}_n^{(h)})$ are given by solving 
\begin{equation}\label{eq:lambda_hat}
\frac{1}{n}\sum_{j=1}^n\mathbf{1}\{V_j>0\}=\widehat{\lambda}_n\int_0^k(1- \widehat{G}_n^{(h)}(x;\widehat{\lambda}_n))\diff x\ ,
\end{equation}
where $\widehat{G}_n^{(h)}(x;\widehat{\lambda}_n)$ is the inversion estimator defined in Section~\ref{sec:estim} with $\lambda=\widehat{\lambda}_n$ and $k>0$ is an additional truncation parameter. Observe that solving \eqref{eq:lambda_hat} entails another step of numerical integration. The convergence rate results of Section~\ref{sec:main} do not carry over directly to the simultaneous estimator. Firstly, there is an additional bias term due to the truncation parameter $k$ of the second integral.  The variance-bias decomposition in this case is no longer straightforward due to the doubly infinite integration domain. Secondly, Proposition~\ref{prop:Z_cov_rate}, which is a crucial step towards Theorem~\ref{thm:gamma_s_rate},  requires modification to take into account the covariance of $\widehat{\lambda}_n$ and the sequence $(Z_1(s),\ldots,Z_n(s))$.  For these reasons, establishing the asymptotic performance of the simultaneous estimator is is left as an open challenge. Nevertheless, simulation experiments suggest that for a large sample, the accuracy of the estimator is close to that of the original estimator that relies on the known arrival rate. 

Figure~\ref{fig:G_lambda} compares the CDF estimators, with and without knowledge of the true $\lambda$, for a mixture of Gamma distributions with $\bm{\alpha}=(1.5,5)$, $\bm{\beta}=(0.8,10)$, and $\bm{p}=(0.4,0.6)$. A sample of $n=40,000$ workload observations was taken, hence Theorem~\ref{thm:rate} prescribes selecting $h=n^{1/5}\approx 8$. Indeed, this choice yields a very good fit when the true $\lambda$ is used. However, for the simultaneous estimator of both the CDF and arrival rate given in \eqref{eq:lambda_hat}, the fit appears to be better for lower values of $h\in[4,6]$. This is likely due to an increased variance of the error in the CDF estimation step. While the estimator with known $\lambda$ clearly outperforms the one without, the latter still yields a good estimator and appears to converge to the true function as the sample size grows (possibly at a slower rate).

 %%%
\begin{figure}[h]
\centering
\includegraphics[width=0.8\textwidth,height=9cm]{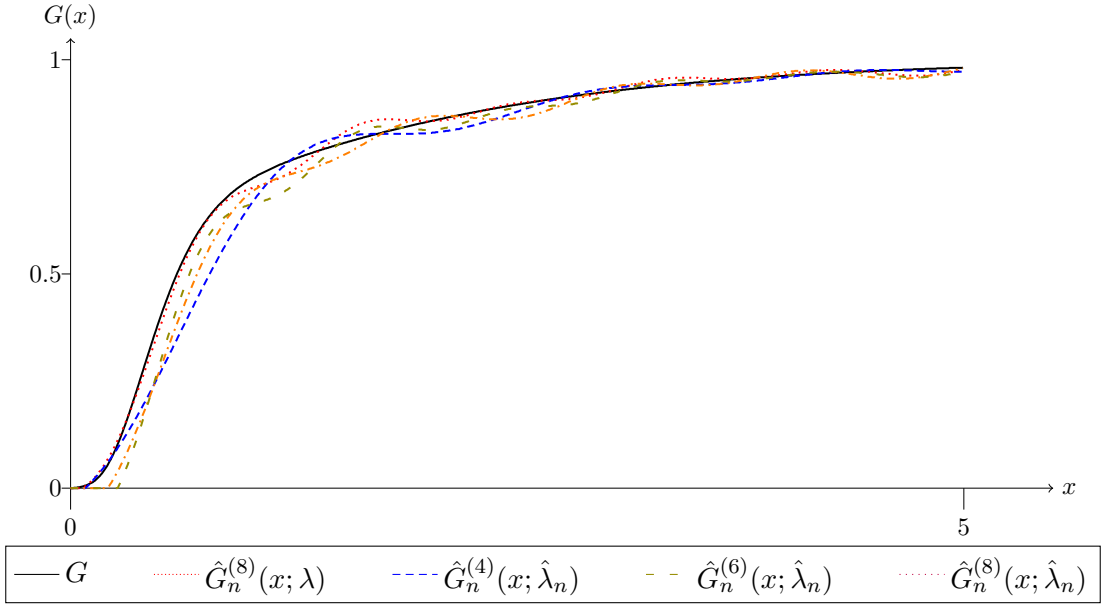}
\caption{The estimators are based on a simulation of $n=40,000$ workload observations with sampling rate $\xi=0.5$. The arrival rate is $\lambda=1$ and the true CDF $G$ (black solid line) is a mixture of gamma distributions $\bm{\alpha}=(1.5,5)$, $\bm{\beta}=(0.8,10)$ and $\bm{p}=(0.4,0.6)$. The true CDF is compared with the estimator with the true $\lambda$ (with $h=8$) and the estimated $\widehat{\lambda}_n$ for $h\in\{4,6,8\}$.} \label{fig:G_lambda}
\end{figure}

Figure~\ref{fig:G_lambda_b} repeats the experiment for a system with lighter load. As opposed the case with known $\lambda$, here the estimation accuracy is actually lower for a system with a lower arrival rate. A possible explanation for this is that the estimation equation \eqref{eq:lambda_hat} becomes flatter, thus resulting in higher bias of the numerical integration step.  

 %%%
\begin{figure}[h]
\centering
\includegraphics[width=0.8\textwidth,height=9cm]
{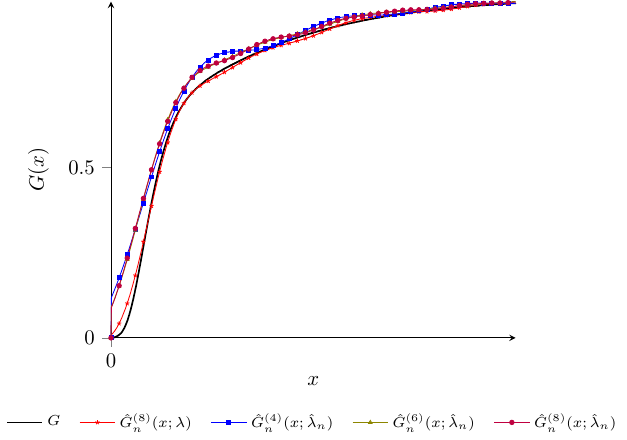}
\caption{The estimators are based on a simulation of $n=40,000$ workload observations with sampling rate $\xi=0.5$. The arrival rate is $\lambda=0.2$ and the true CDF $G$ (black solid line) is a mixture of gamma distributions $\bm{\alpha}=(1.5,5)$, $\bm{\beta}=(0.8,10)$ and $\bm{p}=(0.4,0.6)$. The true CDF is compared with the estimator with the true $\lambda$ (with $h=8$) and the estimated $\widehat{\lambda}_n$ for $h\in\{4,6,8\}$.} \label{fig:G_lambda_b}
\end{figure}

%%%%%%%%%%%%%%%%%%%%%%%%%%%%%%%%%%
\section{Concluding remarks}\label{sec:conclusion}

This work has presented a nonparametric method for the estimation of the job-size CDF in a queue with periodic workload observations. Assuming that the distributions belong to a smoothness class satisfying $ \lim_{s\to\infty} |\gamma(s)|s^{\eta}\leq C_0$, it was shown that the convergence rate of the risk is $O\left(n^{-\eta/(1+\eta)}\right)$. For some distributions in this class, the convergence rate is actually faster due to the faster decay of the bias term. For example,  if $ \lim_{s\to\infty} |\gamma(s)|\e^{-\omega s}\leq C$ for some $\omega>0$ (e.g., the Normal distribution truncated at zero),  then an improved rate can be obtained. Specifically,  applying the arguments used in the proof of Theorem~\ref{thm:rate} it can be shown that setting $h_n=L_w(\sqrt{n})\approx \log \sqrt{n}-\log\log \sqrt{n}$, where $L_w(\cdot)$ is the Lambert W function,  yields a convergence rate of $O\left(L_w(\sqrt{n})^2/n\right)$. An open question that remains is whether the estimator is rate-optimal, in the sense that no other estimator based on the same workload observations achieves a faster rate of convergence. There is no reason to believe this is the case, but verifying that one can actually do better appears to be a very challenging problem due to the complex dependence structure of the data. As far as the author is aware, the estimator for $G$ presented here is the only nonparametric estimator with a rate guarantee for an M/G/1 system with workload observations. An additional issue is that the smoothness parameter $\eta$ of the underlying class of distributions is unknown,  hence the truncation parameter $h_n$ must be selected heuristically in practice. This calls for development of an adaptive estimation technique that learns the smoothness parameter from the data  (see \cite{DGJ2011} for such a scheme in the deconvolution setting).

A direct extension that may be more realistic in some settings is sampling the workload with a probe, which affects performance. For instance, suppose that every probe adds a small deterministic amount of work $d>0$. The only modification required is adapting the conditional CF in \eqref{eq:lapVT} with the previous workload being $V_{j-1}+d$ instead of $V_{j-1}$. The estimator will now include new constants of the form $\e^{id}$, which entail some adaptations to the computations presented here, but this has no effect on the covariance structure of the observations. An additional extension is to allow for deterministic sampling of the workload instead of Poisson sampling. In this case, one cannot use an exact estimation equation because the conditional workload CF based on the previous observation is not tractable. However, this can be overcome using a resampling scheme that mimics a Poisson process, which was introduced in \cite{NMR2024}. There it is shown that the characteristic exponent can be estimated consistently when the observations are equidistant, and this carries over to the CF estimator presented here. Establishing a guarantee on the convergence rate of the covariance terms and, subsequently, the risk in our setting is an open problem for the deterministic sampling scheme.

The approach presented here can also be applied to nonparametric estimation of the job-size density, assuming it exists and satisfies some regularity conditions. An interesting generalization is to extend the analysis to the setting of a Lévy-driven queue, where the input is some general subordinator and not necessarily a compound Poisson process (see \cite{DM2015}). The goal here would be to estimate the underlying Lévy measure of the input. The CF estimator in this case can be obtained directly from \cite{RBM2019}, which dealt with this general setting. However, inverting the net input CF to the Lévy measure is not straightforward. One possibility is to apply the inversion technique used in \cite{G2012}, which entails the estimation of the first and second derivatives of the CF. In the author's view, these directions raise various interesting and challenging questions to be considered in future work. 

Another extension to consider is that of a Markov-modulated input process for which similar estimation equations can be constructed for the Poisson sampling scheme. The main idea here is to utilize the framework of \cite{AK2000} that constructs an exponential (Wald type) martingale for the net-input process where the arrival rate is modulated by a CTMC with transition matrix $Q$ on a finite state space $\{\lambda_1,\ldots,\lambda_p\}$. The CF of the net-input  process can represented by the solution of $|Q+\mathrm{diag}({\boldsymbol \varphi})|=0$, where
\begin{align*}
    \varphi_j(s)=\lambda_j(\gamma(s)-1)-is\ , j=1,\ldots,p\ .
\end{align*}
The methods presented in this work may be applied once the conditional workload distribution at Poisson sampling times is derived. Stochastic root finding will likely be necessary for the matrix inversion described above. An obvious obstacle is that knowledge of the arrival rates and transition matrix, or a reliable  method for their estimation, is necessary.

%%%%%%%%%%%%%%%%%%%%%%%%%%%%%%%%%%
\section*{Acknowledgements}
This work was supported by the Israel Science Foundation (ISF), grant no. 1361/23.  The author thanks the anonymous reviewers for their valuable feedback. The author is grateful to Alexander Goldenshluger and Michel Mandjes for many fruitful discussions and for their helpful feedback on drafts of this work.

%%%%%%%%%%%%%%%%%%%%%%%%%%%%%%%%%%
\bibliographystyle{abbrv}
\small{

}

\end{document}